\documentclass[10pt]{article}

\usepackage{amsfonts}
\usepackage{url}
\usepackage{indentfirst}
\usepackage[small,nohug]{diagrams}
\diagramstyle[labelstyle=\scriptstyle]

\usepackage{xspace}
\usepackage{tikz}
\usetikzlibrary{shapes,arrows}
\usepackage{pgfplots}

\usepackage{amssymb, amsmath, amscd, amsthm}
\usepackage{dsfont}
\usepackage{epsfig}
\usepackage{epic,eepic}
\usepackage{mathrsfs}

\usepackage{verbatim}

\newtheorem{thm}{Theorem}[section]
\newtheorem{lem}[thm]{Lemma}
\newtheorem{prop}[thm]{Proposition}
\newtheorem{cor}[thm]{Corollary}

\newtheorem{innercustomthm}{Theorem}
\newenvironment{customthm}[1]
  {\renewcommand\theinnercustomthm{#1}\innercustomthm}
  {\endinnercustomthm}

\newtheorem{rem}[thm]{Remark}

\newcommand{\preqB}[1][]{{\widehat B(R_{#1})  }}
\newcommand{\preqBB}[1][]{{\widehat B({#1}) }}

\newcommand{\G}[1][\R^{2n} \times S^1] {{\text{\rm Cont}_0(#1)}}
\newcommand{\Gc}[1][\R^{2n} \times S^1] {{\text{\rm Cont}(#1)}}

\newcommand{\GZK}[1][\R^{2n} \times S^1] {{\text{\rm Cont}_0^{\Z_k}(#1)}}
\newcommand{\GZKc}[1][\R^{2n} \times S^1] {{\text{\rm Cont}^{\Z_k}(#1)}}

\newcommand{\R}{{\mathbb{R}}}

\newcommand{\Z}{{\mathbb{Z}}}
\newcommand{\N}{{\mathbb{N}}}
\newcommand{\C}{{\mathbb{C}}}
\newcommand{\CP}{{\mathbb{CP}}}

\newcommand{\Id}{{\mathds{1}}}
\newcommand{\calR}{{\mathcal R}}

\newcommand{\calB}{{\mathcal B}}
\newcommand{\calA}{{\mathcal A}}
\newcommand{\calP}{{\mathcal P}}

\newcommand{\calI}{{\mathcal I}}
\newcommand{\calF}{{\mathcal F}}
\newcommand{\calG}{{\mathcal G}}

\newcommand{\calN}{{\mathcal N}}
\newcommand{\calM}{{\mathcal M}}

\title{Contact non-squeezing at large scale in $\R^{2n} \times S^1$}

\author{Maia Fraser\thanks{Department of Mathematics and Statistics, University of Ottawa}}

\begin{document}

\maketitle

\begin{abstract} 
We define a $\Z_k$-equivariant version of the cylindrical contact homology 
used by Eliashberg-Kim-Polterovich (2006) to prove contact non-squeezing
for prequantized integer-capacity balls $B(R) \times S^1 \subset \R^{2n} \times S^1$, $R \in \N$
and we use it to extend their
result to all $R \geq 1$. Specifically we prove 
if $R \geq 1$ there is no 
$\psi\in \text{Cont}(\R^{2n} \times S^1)$, 
the group of compactly supported contactomorphisms of $\R^{2n} \times S^1$
which squeezes
$\preqB = B(R) \times S^1$
into itself, i.e. maps the closure of $\preqB$ into $\preqB$.
A sheaf theoretic proof of non-existence of corresponding 
$\psi \in \text{Cont}_0(\R^{2n} \times S^1)$,
the identity component
of $\text{Cont}(\R^{2n} \times S^1)$, 
is due to Chiu (2014); it is not known if this is strictly weaker.
Our construction has the advantage of retaining
the contact homological viewpoint of Eliashberg-Kim-Polterovich
and its potential for application in prequantizations of other
Liouville manifolds. 
It makes use of
the $\Z_k$-action generated by a vertical $1/k$-shift but can also
be related, for prequantized balls, to the $\Z_k$-equivariant contact homology developed by
Milin (2008) in her proof of orderability of lens spaces.
\end{abstract}


\section{Introduction}\label{sec:intro}
Gromov's non-squeezing Theorem \cite{gromov} identified a new rigidity 
phenomenon in symplectic geometry: the standard symplectic ball cannot be 
symplectically 
embedded (symplectically squeezed) into any cylinder of smaller radius. 

By contrast, in the contact setting any (Darboux) ball can be contact embedded
into an arbitrarily small 
neighboourhood of a point.
Eliashberg-Kim-Polterovich \cite{ekp} 
therefore studied,
in the contact manifold $\R^{2n} \times S^1$, prequantized balls 
$B^{2n} \times S^1$ and a more
restrictive notion of squeezing: embedding via a globally defined
compactly supported contactomorphism\footnote{
In fact, 
a Darboux ball can be embedded into an arbitrarily small neighborhood of a point
by a compactly supported contact isotopy, and 
 $B^{2n} \times S^1$, $n>1$ can be embedded into 
an arbitrarily small neighborhood of a point by a contactomorphism
but in this case the embedding, while smoothly isotopic to the 
identity, is not 
compactly supported nor contact isotopic to the identity; see \cite{ekp}.}
which can further be required to be isotopic to the identity in this class. 
In their terminology, a {\it (contact) squeezing} of an open set $U_1$
into an open set $U_2$ is 
a compactly supported contact isotopy $\{\phi_t\}_{t\in [0,1]}$ such that 
$\phi_1(\text{\rm Closure}(U_1)) \subset U_2$.
Constructed squeezings in \cite{ekp} are of this kind.
In their proofs of non-existence of a squeezing, however, they 
prove a formally stronger statement, namely,
non-existence of a compactly supported contactomorphism $\psi$ (possibly not isotopic
to the identity) such that $\psi(\text{\rm Closure}(U_1)) \subset U_2$; we will call
such a contactomorphism a \emph{coarse squeezing}.
It is not known in $\R^{2n} \times S^1$ if this formally stronger non-squeezing statement is strictly stronger.

Eliashberg-Kim-Polterovich \cite{ekp} showed contact squeezing is closely related to the 
concept of orderability\footnote{\label{fn:ideal}Given
a Liouville domain $(M, \omega, L)$ with ideal contact boundary $P$ (see Section 1.5 of \cite{ekp} for terminology),
non-squeezing at arbitrarily small scales for ``fiberwise star-shaped domains"
in $M \times S^1$
implies orderability of $P$.}
 of a contact manifold introduced by Eliashberg-Polterovich 
in \cite{elpolterov} and moreover,
squeezing may be possible at one scale and not at another in the same manifold. 
More precisely, consider
$\R^{2n} \times S^1$ with contact structure 
$\ker(dt - \alpha_L)$ and Liouville form $\alpha_L = \frac 1 2 (ydx - xdy)$ on $\R^{2n}$,
$x, y \in \R^{2n}$, $t \in S^1$.
Let $\G$ denote the identity component of the group $\text{Cont}(\R^{2n} \times S^1)$
of compactly supported
contactomorphisms. These are time-1 maps of 
compactly supported contact isotopies, so squeezing of $U_1$ into $U_2$ is equivalent
to the existence of $\phi \in \G$ such that $\phi(\text{\rm Closure}(U_1)) \subset U_2$.
For any positive $R$, let 
$B(R) := \{ w \in \R^{2n}: \pi |w|^2 < R \}$ be the ball of symplectic capacity $R$, and 
$\preqB := B(R) \times S^1$ its 
prequantization. Eliashberg-Kim-Polterovich proved:

\begin{thm}[Eliashberg-Kim-Polterovich \cite{ekp}]\label{thm:ekp}
Let $R < 1.$ Then there is a contact squeezing of $\preqB$ into itself.
By contrast, if $R \geq 1$ is an integer, there is no coarse contact squeezing of $\preqB$ into itself.
\end{thm}

It remained unknown for some time whether non-squeezing 
in this setting held also for non-integer $R > 1$. In 2010 Tamarkin \cite{tamarkin} sketched a
proof of the affirmative answer. This was recently proved by Chiu:

\begin{thm}[Chiu \cite{chiu}]\label{thm:main}
Let $R \geq 1$. Then there is no contact squeezing of $\preqB$ into itself.
\end{thm} 

\medskip
Discussions with Tamarkin \cite{tamarkin} also inspired our use of $\Z_k$-equivariance:
we prove Theorem~\ref{thm:main} by means of a $\Z_k$-equivariant version of 
the cylindrical contact homology of Eliashberg-Kim-Polterovich \cite{ekp}. In fact, as in \cite{ekp},
we prove a formally stronger statement:
\begin{customthm}{\ref{thm:main}'}\label{thm:main-strong}
Let $R \geq 1$. Then there is no coarse contact squeezing of $\preqB$ into itself.
\end{customthm}
In Section~\ref{sec:room}, we observe that Theorem~\ref{thm:main} also has consequences for
squeezings of $\preqB$ into itself when $R < 1$: in some cases
they require a larger
domain of support, i.e., {\it squeezing room},
than previously established. This is stated in Theorem~\ref{thm:main2}.

\begin{rem}\label{rem:into-itself}
Squeezing of $\preqB$ into itself is
an open condition.
Indeed, existence of a squeezing
of $\preqB$ into itself implies
existence of a squeezing of some larger $\preqB[1]$
into a smaller $\preqB[2]$ and this
will in particular squeeze all intermediate prequantized balls into
themselves. 
An equivalent\footnote{
Theorems~\ref{thm:main} and \ref{thm:main}'
refer
 to existence of a squeezing only;
the required squeezing room could in principle
be greater the farther apart
$R_1$ and $R_2$ are taken.}
formulation of Theorem~\ref{thm:main} (resp. \ref{thm:main}')
is therefore : let $1 < R_2 < R_1$, then there is no 
squeezing (resp. coarse squeezing) 
of $\preqB[1]$ into $\preqB[2]$. Note the strict inequality $1 < R_2$
which can without loss of generality be imposed.
\end{rem}

\tableofcontents

\section{Squeezing vs. $\Z_k$-equivariant squeezing} \label{sec:zk-action}

We prove Theorem~\ref{thm:main-strong} by proving an alternate,
equivalent statement, Theorem~\ref{thm:cover}.
To formulate this, let $\GZK$ denote the identity 
component of the group $\text{Cont}^{\Z_k}(\R^{2n} \times S^1)$ of compactly supported
$\Z_k$-equivariant contactomorphisms for the $\Z_k$-action 
generated by a vertical shift $\nu: (x, y, t) \mapsto (x, y, t + 1/k)$.  
By analogy with the definition of contact squeezing \cite{ekp},
define a \emph{$\Z_k$-equivariant (contact) squeezing} of an 
open set $U_1$ into an open set $U_2$ as 
$\phi \in \GZK$ such that $\phi(\text{Closure}({U_1})) \subseteq {U_2}$.
Define a \emph{coarse $\Z_k$-equivariant squeezing} analogously but requiring
only $\phi \in \GZKc$.
Our main result is:

\begin{thm}\label{thm:cover}
For any prime $k \in \N$ and any $\ell \in \N$ such that $\ell < k$, if
$R_2 < 1/\ell < R_1$,
there is no $\Z_k$-equivariant contact squeezing of $\preqB[1]$ into $\preqB[2]$, not even a coarse one. 
\end{thm}

\begin{rem}\label{rem:kgreater2}
Note that case $k=2$ of Theorem~\ref{thm:cover} is implied by 
the non-existence of a squeezing 
of $\widehat B(1)$ into itself (proved by  \cite{ekp}, see
Theorem~\ref{thm:ekp} above). It will therefore be sufficient for the purposes
of proving Theorem~\ref{thm:cover}
to establish the case $k>2$ and we
make this assumption in the 
computations of Section~\ref{sec:groups-preqB}. 
\end{rem}

The equivalence of Theorems~\ref{thm:main-strong} and \ref{thm:cover}  
 follows from properties of
 the contact $k$-fold cover of $\R^ {2n} \times S^1$.
Indeed, let $S^1 = \R / \Z$ and, as a manifold, 
define the $k$-fold cover of $\R^ {2n} \times S^1$ 
by the
covering map
\begin{align*}
\tau: \R^ {2n} \times S^1 &\to \R^ {2n} \times S^1 \\
(z, t) &\mapsto (\sqrt k z, kt).
\end{align*} 
Assuming the standard contact structure on both base and cover, one
has that $\tau$ is a contactomorphism.
Deck transformations in the cover are then also
contactomorphisms; they form 
a cyclic group isomorphic to $\Z_k$, generated by $\nu$. We have:

\begin{lem}\label{lem:equiv}
$\GZKc = \{ \tilde\phi : \phi \in \Gc \}$ where 
$\tilde\phi$ is the unique lift (as a compactly supported diffeomorphism) of $\phi$.
Likewise we have $\GZK
= \{ \tilde\phi : \phi \in \G \}$.
Moreover, 
for any $R_1, R_2 > 0$, 
$\tilde\phi(\widehat B(R_1/k)) \subseteq \widehat B(R_2/k)
\Leftrightarrow \phi(\widehat B(R_1)) \subseteq \widehat B(R_2)$.
\end{lem}

\begin{proof}
Each $\phi \in \Gc$ has a unique lift\footnote{For shorthand, write $\tau: C\to B$
where both cover $C$ and base $B$ are $\R^{2n} \times S^1$
and $\phi:B\to B$ is a compactly supported diffeomorphism.
Lift $\phi$ to the local diffeomorphism $\hat \phi:C \to B$ 
given by $\hat \phi := \phi \circ \tau$. Because 
$B=\R^{2n} \times S^1$ and $\phi$ has compact support, $\phi$ acts trivially
on $\pi_1(B)$ so $\hat \phi$ maps $\pi_1(C)$ to the 
same subgroup of $\pi_1(B)$ as does $\tau$, namely $k\Z \subset \Z = \pi_1(B)$ and, hence, 
by the unique lifting property,
lifts to $\tilde \phi:C \to C$ such that
$\tau \tilde \phi(z) = \hat \phi(z) = \phi \tau (z)$ 
which is unique up to composition with deck transformations;
of such lifts there is a unique one with compact support.}
to a compactly supported diffeomorphism $\tilde\phi$ of the $k$-fold cover 
and this
$\tilde\phi$ is a contactomorphism
since $\phi$ is and the contact structure in the cover is the pullback by $\tau$ of its 
counterpart in the base. 
Moreover, $\tilde\phi$ is $\Z_k$-equivariant by construction since $\nu$ 
is a deck transformation. 
Thus
$\tilde\phi \in \GZKc$. On the other hand any element $\phi' \in \GZKc$
descends to a well-defined contactomorphism $\phi$ of the base since it 
commutes with $\nu$, 
and so we have $\phi' = \tilde \phi$. Applying this correspondence
to isotopies we obtain the statement regarding
$\GZK$ and $\G$.
The final statement of the Lemma is immediate because 
$\tilde\phi(\tau^{-1}(\mathcal U)) \subseteq \tau^{-1}(\mathcal V)
\Leftrightarrow \phi(\mathcal U) \subseteq \mathcal V$.
\end{proof}

As a Corollary, we obtain the claimed equivalence of 
Theorems~\ref{thm:main-strong} and ~\ref{thm:cover}:

\begin{cor}\label{cor:equivalent}
$\preqB$ can be squeezed into itself by $\phi \in \Gc$ 
if and only if $\widehat B(R/k)$ can be squeezed into itself by 
$\tilde\phi \in \GZKc$ (and likewise for $\G$ and $\GZK$).
Thus, Theorem~\ref{thm:main-strong} is equivalent to Theorem~\ref{thm:cover}.
\end{cor}

\begin{proof}The first statement is immediate from
Lemma~\ref{lem:equiv}. 
For the equivalence of Theorems note that, once 
$k$ prime and $\ell < k$ are fixed, $R_2$ in 
the hypotheses of
Theorem~\ref{thm:cover} can without loss of generality be specified
to be arbitrarily close to $1/\ell$, in particular such that 
$1/k < R_2  < 1/\ell < R_1$
since squeezing $\preqB[1]$ into a smaller $\preqB[2]$ would imply
squeezing $\preqB[1]$ into all $\preqB$ for $R_2 \leq R < R_1$
Therefore,
Theorem~\ref{thm:cover}
is equivalent (putting $R_i' = k R_i$)
to non-existence - 
under the hypothesis $k$ prime, $\ell < k$,
 and $1 < R_2'  < k/\ell < R_1' $ - 
of $\phi \in \Gc$ which squeezes
$\widehat B(R_1')$ into $\widehat B(R_2')$. 
Since for every pair 
$R_2' < R_1'$ there exist\footnote{
Put $\beta = R_1/R_2$ and consider intervals $I_\ell = (\ell R_2, \ell \beta R_2)$, $\ell \in \N$. 
Since $\beta>1$, for sufficiently large $\ell$ each $I_\ell$ overlaps with $I_{\ell+1}$
and so by the infinitude of primes $\exists \ell \in \N$ such that $I_\ell$
contains a prime $k$; i.e., $R_2 < k/\ell < R_1$.} $k, \ell \in \N$ with $k$ prime such that $R_2' < k/\ell < R_1'$
the above can be restated  
as non-existence of a squeezing of $\widehat B(R_1')$ into $\widehat B(R_2')$ when $1 < R_2' < R_1'$. 
By Remark~\ref{rem:into-itself} this is equivalent to Theorem~\ref{thm:main-strong}.
\end{proof}

\begin{rem}[The implicit role of $\Z_k$-equivariance in \cite{ekp}] 
Eliashberg-Kim-Polterovich \cite{ekp} use (non-equivariant) cylindrical contact homology to
prove that $\widehat B(1)$ cannot be squeezed into itself.
They then conclude by a covering space argument that $\widehat B(m)$ cannot be squeezed into itself
for any $m \in \N$. Their two-part proof does not emphasize the inherent $\Z_m$-action;
however, in the language of the present section, it amounts to
first showing $\widehat B(1)$ cannot be squeezed into itself,
noting this implies $\widehat B(1)$ cannot be squeezed into itself $\Z_m$-equivariantly,
and then using the last statement of Lemma~\ref{lem:equiv} to
conclude $\widehat B(m)$ cannot be squeezed into itself.
For the prequantized ball $\widehat B(1)$,
both $\Z_k$-equivariant squeezing and (non-equivariant) squeezing 
are impossible and, moreover, $\widehat B(1)$ is the only 
prequantized ball for which the contact homology of \cite{ekp}
can directly\footnote{The $\Z$-graded vector spaces $CH_*(\preqB)$ are isomorphic for all $R> 1$ and
differ  
from $CH_*(\widehat B(r))$ for $r < 1$. This makes 
the invariants suitable to directly rule out only squeezing of $\preqB$ into $\widehat B(r)$ 
(since we allow squeezings to have arbitrarily large support - see Remark~\ref{rem:compareNS}).} 
rule out squeezing. A priori, however,
non-existence of a $\Z_m$-equivariant squeezing is a weaker notion
than non-existence of a squeezing and could potentially be 
true in more situations. Indeed, this is the 
contribution of the present paper. For $k$ prime,
we use $\Z_k$-equivariant contact homology
 to rule out
$\Z_k$-equivariant squeezing of any $\widehat B(1/\ell)$, $\ell \in \N$ into itself
when $\ell < k$ (Theorem~\ref{thm:cover}), although such prequantized balls
are known by \cite{ekp} to be squeezable into themselves 
non-equivariantly (c.f. Theorem~\ref{thm:ekp}). 
\end{rem}

\section{$\Z_k$-equivariant contact homology $CH_{*}^ {\Z_k}(-)$}


We now define $CH_{*} ^{\Z_k}(-)$, 
a $\Z_k$-equivariant analog of the (non-equivariant) cylindrical
contact homology  $CH_*(-)$
developed by Eliashberg-Kim-Polterovich \cite{ekp} and
Kim \cite{kim} for $\R^{2n} \times S^1$,
which also has similarities to a $\Z_k$-equivariant version
of $CH_*(-)$ developed by Milin \cite{milin} for a different $\Z_k$-action.
Like these theories, it lies within the general framework of symplectic field theory
proposed by Eliashberg-Givental-Hofer \cite{egh}.

We give the construction of $CH_{*} ^{\Z_k}(-)$ as
we recall that of $CH_*(-)$, and we explain
how well-definedness of $CH_{*} ^{\Z_k}(-)$ follows with only minor
modifications from corresponding arguments for $CH_*(-)$
due to \cite{beh3}, \cite{dragnev}, \cite{bourgeois}, \cite{hwz} with,
in addition, the construction of coherent orientations 
from Bourgeois-Mohnke \cite{fred-klaus}; these
aspects are highlighted in boldface in the construction
and addressed in the paragraph ``Technical arguments".

Foundational issues with 
cylindrical contact homology 
which arise in the presence of multiply covered orbits (see for example \cite{jo}
for a discussion) 
are avoided in both our 
setting and that of Eliashberg-Kim-Polterovich \cite{ekp}
by taking as generators only 
closed Reeb orbits in the free homotopy class $[ \text{pt} \times S^1]$.

In Section~\ref{sec:results-preqB} we state our main results
for $CH_{*}^{\Z_k}(\preqB)$, showing how they imply
non-squeezing. In Section~\ref{sec:groups-preqB}
we give proofs of these statements as well as a proof
of analogous statements for $CH_{*}(\preqB)$ with $\Z_k$-coefficients. 
Our aim is for the present 
paper to serve as an extension of Eliashberg-Kim-Polterovich
\cite{ekp} in which the reader familiar with \cite{ekp} 
can view non-squeezing at large scale 
from the vantage point provided by contact homology.

\begin{rem}\label{rem:Z2vsZk}
To keep the presentation compact: {\bf we assume $k > 2$ from now on}. Our construction
applies equally well to the case $k=2$ but this would
require specification of a
different projective resolution to compute equivariant homology
(see Remark~\ref{rem:equivHom}) and result in slightly 
different\footnote{In
all these complexes multiplication by $(T-1)$ must be
replaced with multiplication by $(T+1)$.}
chain complexes in equivariant computations. Since we do not need the case $k=2$
(see Remark~\ref{rem:kgreater2}) we omit it.
\end{rem}

We restrict attention to $V := \R^{2n} \times S^1$
with contact structure $\ker(dt - \alpha_L)$ 
and $\Z_k$-action as defined in Sections~\ref{sec:intro} and \ref{sec:zk-action}. Our 
construction, however, like that of \cite{ekp}, goes through verbatim
in prequantizations $M \times S^1$ for many other Liouville manifolds $M$
(c.f. Theorem 4.47 of \cite{ekp}).

An open domain $U \subset V$
with compact closure is said to be 
{\it fiberwise star-shaped} \cite{ekp} if its boundary $\partial U$ is transverse to the fibers
$M \times \{t\}$, $t \in S^1$ and intersects them along hypersurfaces transverse to
the Liouville vector field $L$ determined by $\alpha_L = i_L\omega$. In particular
prequantized balls $\preqB$ are fiberwise star-shaped.
Let $\mathscr{U}$,
resp. $\mathscr{U}_k$,
be the class of domains $\psi(U)$ such that $U$ is fiberwise star-shaped,
resp. fiberwise star-shaped and $\Z_k$-invariant, and 
$\psi \in \Gc$, resp. $\psi \in \GZKc$.
Given $U \in \mathscr{U}_k$, we construct the
$\Z$-graded vector space $CH_*^{\Z_k}(U)$,
as \cite{ekp}
constructed $CH_*(U)$ for $U \in \mathscr{U}$,
so that the resulting association is functorial in 
an invariant way
(invariant under the action of $\GZKc$ resp. $\Gc$ - 
see Theorem~\ref{thm:g-functor} below, resp.
Theorem 4.47 of \cite{ekp}).

\smallskip
{\bf Admissible forms.} 
Denote by $\calF_{ad}(U)$ the set of all {\it admissible}
contact forms on $U$, namely forms
$\lambda = F(dt - \alpha_L)$ 
where $F$ is a positive Hamiltonian
equal to a constant $K$ outside a compact set,
and the Reeb flow of $\lambda$ has
no contractible closed orbits $\gamma$ of 
action (i.e. period) $\calA(\gamma) := \int_{\gamma}\lambda$ less
than or equal to $K$. 
Let 
$\calF_{ad}(U, \epsilon) \subset \calF_{ad}(U)$
consist of those forms which do not have $\epsilon$ as {\it critical value},
i.e. have no closed Reeb orbit $\gamma$ of action $\calA(\gamma) =\epsilon$.  Denote
by $\calF^{\Z_k}_{ad}(U, \epsilon)$ those $\lambda \in \calF_{ad}(U, \epsilon)$ which are 
$\Z_k$-invariant.
As in \cite{ekp},
endow these spaces of contact forms with the ``anti-natural" partial order $\preceq$:
$$\lambda''\preceq \lambda' \Leftrightarrow \lambda ''\geq \lambda'.$$
An admissible contact form 
$\lambda$ equal to $K(dt -\alpha_L)$ outside a compact set is 
said to be {\it regular} if all orbits $\gamma$ in
$$\calP_\lambda := \{\gamma : \calA(\gamma) < K, [\gamma] = [\text{pt}\times S^1]\}$$
are {\it non-degenerate}. Here 
$[\gamma]$ denotes the free homotopy class of $\gamma$ as a loop $S^1 \to V$
and non-degeneracy means the restriction of the Poincar\'e return map for the Reeb 
flow of $\lambda$ along $\gamma$ to the contact hyperplane bundle does not have 1 as eigenvalue.
By standard arguments generic admissible contact forms
are {regular}. 

For regular $\lambda$, $\calP_\lambda$ is graded by the
 Conley-Zehnder index
$\mu_{CZ}(\gamma)$ defined in terms of
paths of symplectic matrices as in \cite{robbinsalamon}. 
Note this convention differs by $-n$ from that in \cite{ekp} (c.f. Erratum to \cite{ekp}). 
If $\lambda$ is $\Z_k$-invariant, then $\gamma \mapsto \nu\gamma$
generates a $\Z_k$-action on $\calP_\lambda$ which
preserves $\mu_{CZ}$.
Using $\lambda$, identify 
the symplectization $W$ of $V$ with $(V \times \R, d(e^s \lambda))$
and put $W_\ast := (V\setminus \{0\} \times S^1) \times \R \subset W$.
Let $\xi$ and $\tau$ denote the hyperplane bundles on $W_\ast$ resp. $W$ which are
respectively pull-backs of the standard contact structure on $S^{2n-1}$ and
$\ker(\lambda)$ on $V$.
Lift also the 
$\Z_k$-action of $V$ to $W$, $W_\ast$.
Write $R_\lambda$ for
the Reeb vector field of $\lambda$.
Consider
almost complex structures $J$ on $V \times \R$ which are {\it adjusted} to $\lambda$
in the sense of \cite{ekp}: $J$ is invariant under translations in the $\R$-coordinate,
$J\tau = \tau$ and $J|_\tau$ is compatible with $d\lambda$, 
$J(\frac {\partial} {\partial s}) = R_\lambda$, 
and finally, $J\xi = \xi$ outside the symplectization
of a compact subset of $V$.
When $\lambda$ is $\Z_k$-invariant, a
$J$ adjusted to $\lambda$ which is also $\Z_k$-invariant is
said to be {\it $\Z_k$-adjusted} to $\lambda$.

\smallskip
{\bf Chain complex.} 
Given an almost complex structure $J$ adjusted, resp. $\Z_k$-adjusted, to 
a regular admissible $\lambda$,  
let
$C(\lambda, J)$, 
resp. $C^{\Z_k}(\lambda, J)$, be the vector space generated over $\Z_k$
 by orbits in $\calP_\lambda$ and $\Z$-graded by $\mu_{CZ}$.
 For $a, b$ which are not critical values of $\lambda$, using the action filtration
 let 
 $$C^{(a, b)}(\lambda, J) = C(\lambda, J) \cap \text{span}\{\gamma: \calA(\gamma) < b\}/
 \text{span}\{\gamma : \calA(\gamma) > a\}$$
and define $C^{\Z_k, (a, b)}(\lambda, J)$ analogously.
To define a differential $d$ we make no changes to the definition
of \cite{ekp}, however {\bf transversality} arguments needed to establish
well-definedness of $d$ for generic $J$
require a slight modification as explained below. Note: the definition of
the differential in \cite{ekp}
is given in more generality,
as part of the construction of generalized Floer homology; for 
a more accessible, $CH$-specific description
the reader is referred to \cite{bourgeois-summer}.
Roughly speaking, for both equivariant and non-equivariant
theories we consider for closed Reeb orbits $\gamma_\pm$
the moduli space $\hat \calM(\gamma_+, \gamma_-)$ consisting of
$J$-holomorphic cylinders 
$$G = (g, a): (\R \times S^1, j) \to (W = V \times \R, J)$$
asymptotic at the ends to $\gamma_\pm$
such that
$\lim_{s \to \pm \infty}f(s, t) = \gamma_\pm(T_\pm t)$
and $\lim_{s \to \pm \infty} \allowbreak a(s, t) = \pm \infty$, for $T_\pm$
the periods of $\gamma_\pm$ respectively, and assuming
standard complex structure $j$ on $\R \times S^1$
and almost complex structure $J$ adjusted to $\lambda$ on $W$.
In the non-equivariant case, for generic almost complex structures
adjusted to $\lambda$
 the moduli space $\hat \calM(\gamma_+, \gamma_-)$ 
forms a smooth oriented manifold.
In the equivariant case, we allow only
$\Z_k$-adjusted complex structures and 
modify the
usual genericity argument as explained below.

This manifold $\hat \calM(\gamma_+, \gamma_-)$, 
in both equivariant and non-equivariant cases, is acted on by the group
$\R \times (\R \times S^1)$ of holomorphic
re-parametrizations (of target, domain) and quotienting by these one 
obtains a smooth oriented manifold $\calM(\gamma_+, \gamma_-)$ 
 of dimension 
$\mu_{CZ}(\gamma_+) - \mu_{CZ}(\gamma_-) - 1$. 
Since $J$ and $\lambda$ are 
standard at infinity and $\lambda$ has no contractible
closed Reeb orbits with action $\leq K$, 
{\bf compactness} results of \cite{beh3} show that 
$\calM(\gamma_+, \gamma_-)$ compactifies to a moduli space of 
``broken"
$J$-holomorphic cylinders; this applies 
in both the equivariant and non-equivariant frameworks.
In the case $\mu_{CZ}(\gamma_+) - \mu_{CZ}(\gamma_-) = 1$
this moduli space
is a compact, oriented $0$-dimensional manifold and so
consists of a finite number of points with sign. 
For this, one needs to have arranged a system of {\bf coherent orientations}
on moduli spaces in a way which is compatible with glueing,
and, in the case of our $\Z_k$-equivariant theory, invariant
under the action of $\Z_k$. This extra property comes for free
from the construction of \cite{fred-klaus} assuming
the contact forms and almost complex structures used
are $\Z_k$-invariant (see below). The differential on $C(\lambda, J)$,
resp. $C^{\Z_k}(\lambda, J)$ is then defined by counting, 
modulo $k$, all points in $\calM(\gamma_+, \gamma_-)$ with signs.
By construction of the compactification \cite{hwz}
it follows that $d^2 = 0$ so $(C(\lambda, J), d)$,
resp. $(C^{\Z_k}(\lambda, J), d)$, is a chain complex. 
Moreover, since $G \in \calM(\gamma_+, \gamma_-)$ if and only 
if $\nu G \in \calM(\nu \gamma_+, \nu \gamma_-)$ 
and signs of elements of $\calM(\gamma_+, \gamma_-)$
are preserved by $\nu$, it follows that $(C^{\Z_k}(\lambda, J), d)$
comes equipped with a $\Z_k$-action, i.e., $d$ is $\Z_k$-equivariant
for the $\Z_k$-action induced on spaces of Reeb orbits by $\nu$. $CH_*^{(a, b)}(\lambda)$
is defined as the homology of $(C^{(a, b)}(\lambda, J), d)$ after showing
it does not depend on $J$. Likewise we define
$CH_*^{\Z_k, (a, b)}(\lambda)$
as the $\Z_k$-equivariant homology of $(C^{\Z_k, (a, b)}(\lambda, J), d)$.
Note that any contactomorphism, resp. 
$\Z_k$-equivariant contactomorphism, $\psi$
will set up a 1-1 correspondence 
not only between chain groups
for $\lambda$ and $\psi_*\lambda$ but also 
between moduli spaces defined above
so there is a chain map, resp. $\Z_k$-equivariant chain map,
$\psi_\sharp$ from the respective chain complex for $\lambda$ to 
that for $\psi_*\lambda$ which is an isomorphism in all degrees,
yielding (grading-preserving) 
isomorphisms:
$\psi_\sharp: CH_*^{(a, b)}(\lambda) \xrightarrow{\cong} CH_*^{(a, b)}(\psi_*\lambda)$
and 
$\psi_\sharp: CH_*^{\Z_k, (a, b)}(\lambda) \xrightarrow{\cong} CH_*^{\Z_k, (a, b)}(\psi_*\lambda).$

\begin{rem}{\bf (Equivariant homology computation)} \label{rem:equivHom}
In general, given a $\Z_k$-action on a chain complex $(C_*, d)$,
$\Z_k$-equivariant homology is defined as follows.
Let
$\calR = \Z_k[\Z_k] \cong \Z_k[T]/(T^k-1)$ be the group ring of the group $\Z_k$ with $\Z_k$-coefficients.
Let $\Z_k$ act on
$\Z_k$ trivially to make
$\Z_k$ into an $\calR$-module. 
Let $(E_*, \delta)$ be any projective resolution
of $\Z_k$ (as an $\calR$-module) 
and tensor $(E_*, \delta)$ with $(C_*, d)$ over $\calR$. The $\Z_k$-equivariant
homology of $(C_*, d)$ is defined to be the usual homology
of this tensor product. Up to isomorphism this is independent 
of the choice of $(E_*, \delta)$
since all projective resolutions are quasi-isomorphic 
as $\calR$-chain complexes (to $0\to \calR \to 0$ and so to each other) and 
thus the resulting tensor products are also quasi-isomorphic.
In computations, we follow Milin \cite{milin} and - assuming $k > 2$ - use the projective resolution
$$\ldots \calR \xrightarrow{\cdot(T-1)}\calR \xrightarrow{\cdot(T^{k-1}+\ldots+1)}\calR \xrightarrow{\cdot(T-1)}\calR \xrightarrow{\cdot(T^{k-1}+\ldots+1)}\calR \xrightarrow{\cdot(T-1)}\calR \xrightarrow{}0.$$
\end{rem}

\smallskip
{\bf Monotonicity morphisms.}
Given $\lambda_- < \lambda_+$, \cite{ekp} define\footnote{(translated
to contact structures from the language of Hamiltonian structures
in \cite{ekp})}  a
monotonicity chain map
$\textrm{mon}: C^{(a, b)}(\lambda_+, J) \to C^{(a, b)}(\lambda_-, J)$
by considering the moduli space $\hat \calM(\gamma_+, \gamma_-)$ of 
$J$-holomorphic cylinders between $\gamma_\pm$ which are
closed Reeb orbits for $\lambda_\pm$ respectively, where $J$ is
an almost complex structure ``adjusted to an admissible concordance
structure" between $\lambda_\pm$. This construction is then extended
to the case $\lambda_- \leq \lambda_+$, i.e., $\lambda_+ \preceq \lambda_-$.
In our equivariant setting
we follow the identical procedure using however $\Z_k$-invariant ingredients.

To describe this more precisely,
 assume $\lambda_- < \lambda_+$ 
 where $\lambda_\pm = F_\pm(dt - \alpha_L)$,
$F_\pm > 0$. Let $a_- < a_+ \in \R$ and
put $W_{a_-}^{a_+} = V \times [a_-, a_+]$ 
with coordinates $(v, s)$. Consider any function $F: W_{a_-}^{a_+} \to \R_+$ 
such that $\frac {\partial F}{\partial s} > 0$, 
and, outside
a compact subset of $V$, $F$ depends only on $s$, while near the 
respective boundaries, $s = a_\pm$, $F$ is
linear in $s$ 
of the form $F(v, s) = (1 + s - a_\pm)F_\pm(v)$.
In particular, $F(v, a_\pm) = F_\pm(v)$.
Using such an $F$, $W_{a_-}^{a_+}$ can be identified with 
the region $\{ F_-(v) \leq r \leq F_+(v) \}$,
$(v, r) \in V \times \R_+ = W$ via the map $\Phi_F: 
(v, s) \mapsto (v, F(v, s))$.
View $V \times \R_+$ as the symplectization of $V$
with symplectic form $d(r(dt -\alpha_L))$
and denote the 
pull-back of this 2-form to $\text{Int}(W_{a_-}^{a_+})$ as
$\omega_F := d(F(dt - \alpha_L))$.
The pair $(W_{a_-}^{a_+}, F(dt - \alpha_L)$ is said to be an {\it (admissible)
concordance structure} between $\lambda_\pm$, or 
in
the equivariant setting an
{\it (admissible)
$\Z_k$-concordance structure} if $F$ is also $\Z_k$-invariant
(for the lifted $\Z_k$-action). 
The linearity of $F$ in $s$ near the boundaries,
$s = a_\pm$, of $W_{a_-}^{a_+}$ (c.f. ``normal form" in \cite{ekp}) gives neighbourhoods
of these boundaries already locally the
structure of symplectizations of $(V, \lambda_\pm)$ respectively.
By extending $F$ linearly to all of $\R$ and extending $\omega_F$
accordingly, $(\overline W = V \times \R, \omega_F)$ acquires
the structure of a symplectic cobordism between $(V, \lambda_-)$
and $(V, \lambda_+)$. Moreover, when
$F$ is $\Z_k$-invariant
so is the symplectic structure $\omega_F$ on $\overline W$.
Warning: though $W = \overline{W} = V \times \R$, 
we use different names to recall the different
symplectic structures  -
$(\overline W, \omega_F)$ is a cobordism between $(V, \lambda_\pm)$ while
$(W, d(e^s\lambda))$ is the symplectization of $(V, \lambda)$. 

An almost complex structure $J$ on $W_{a_-}^{a_+}$ is {\it adjusted}
to the concordance $(W_{a_-}^{a_+}, \allowbreak F(dt - \alpha_L))$ 
between $\lambda_\pm$ for specific choices, $J_\pm$,
of respectively adjusted almost complex structures on 
symplectizations $(W, d(e^s\lambda_\pm))$ of $(V, \lambda_\pm)$
 if $\omega_F$ tames $J$, $J$
is pseudoconvex at infinity and $J$ agrees with $J_\pm$
near the boundaries $s = a_\pm$. The second condition means
that $\overline{W}$
is foliated by weakly $J$-convex hypersurfaces
outside the symplectization of a compact subset of $V$. 
In the equivariant setting, an adjusted
$J$ which is also $\Z_k$-invariant is said to be {\it $\Z_k$-adjusted}.
In either case, we denote also by $J$ the extension of 
$J$ to $\overline W$.
Analogous to when we defined the differential $d$, 
given $\lambda_- < \lambda_+$
 we consider in both equivariant and non-equivariant theories
the moduli space $\hat \calM(\gamma_+, \gamma_-)$ consisting of
$J$-holomorphic cylinders 
$$G = (f, a): (\R \times S^1, j) \to (\overline{W} = V \times \R, J)$$
asymptotic at the ends to $\gamma_\pm$
such that
$\lim_{s \to \pm \infty}f(s, t) = \gamma_\pm(T_\pm t)$
and $\lim_{s \to \pm \infty} \allowbreak a(s, t) = \pm \infty$, for $T_\pm$
the periods of $\gamma_\pm$ respectively, and assuming
standard complex structure $j$ on $\R \times S^1$
and almost complex structure $J$ adjusted to a
chosen concordance between $\lambda_\pm$.
In the non-equivariant setting
standard transversality 
arguments (as in \cite{dragnev}, \cite{bourgeois}) establish
that for generic $J$ the space
$\hat \calM(\gamma_+, \gamma_-)$ is a smooth oriented manifold
(assuming
once again a system of coherent orientations).
Slight modifications of these
arguments (see below) apply in our equivariant setting
implying the same statement for generic
almost complex structures $\Z_k$-adjusted 
to a $\Z_k$-concordance.

The manifold $\hat \calM(\gamma_+, \gamma_-)$
is acted upon freely by the re-parametrization\footnote{(there 
is no longer invariance of $J$ in the $s$-direction)} 
group $\R \times S^1$ and quotienting yields a smooth
manifold $\calM(\gamma_+, \gamma_-)$ of dimension $\mu_{CZ}(\gamma_+)
-\mu_{CZ}(\gamma_-)$, which compactifies to a moduli 
space of broken $J$-holomorphic cylinders by \cite{beh3}.
When $\mu_{CZ}(\gamma_+)=\mu_{CZ}(\gamma_-)$ the space
$\calM(\gamma_+, \gamma_-)$ 
is a finite collection of points with sign which we count modulo $k$ 
and the chain map $\textrm{mon}:C^{(a, b)}(\lambda_+, J) \to C^{(a, b)}(\lambda_-, J)$
is defined, exactly as in \cite{ekp}, by setting
$\textrm{mon}(\gamma_+)$ to be
the sum over all $\gamma_-$ weighted by this count 
$\#\calM(\gamma_+, \gamma_-)$. This induces the
(grading-preserving) monotonicity
morphism
$\textrm{mon}: CH^{\Z_k, (a, b)}(\lambda_+, J) \to CH^{\Z_k, (a, b)}(\lambda_-, J)$.

Three natural grading-preserving morphisms, besides $\textrm{mon}$, are important - those
due to scaling invariance, contactomorphism invariance ($\psi_\sharp$) and window enlargement.
We've given $\psi_\sharp$. The other two are also
immediate in both equivariant and non-equivariant settings from the 
definitions. Assume $c > 1$, then there are chain maps
\begin{align*}
c_*&: C^{\Z_k, (0, cb)}(c\lambda, J_c) \xrightarrow{\cong} C^{\Z_k, (0, b)}(\lambda, J) \\
\textrm{win}&: C^{\Z_k, (0, b)}(\lambda, J) \to C^{\Z_k, (0, cb)}(\lambda, J)
\end{align*}
where the second map is an isomorphism in all gradings
under the additional hypothesis that $\lambda$ 
has no closed Reeb orbits with action in the window $[b, cb]$, and
$J_c$ denotes the re-scaled version of $J$ adjusted to $c\lambda$.
In general, the composition of $c_*$ and $\textrm{win}$,
in either order, gives $\textrm{mon}$ (see Remark 4.40 in \cite{ekp}
which applies verbatim in our setting). In particular, this allows
to extend the definition of $\textrm{mon}$ to the case
$\lambda_+ \geq \lambda_-$ by first scaling $\lambda_+$ by 
a suitably small factor $c >1$.
Moreover, by definition, concordances behave well under 
compactly supported contactomorphisms $\psi$ since 
these preserve the contact forms at infinity and thus induce 
a bijective correspondence between concordances. Indeed,
given contact forms $\lambda_\pm= H_\pm(dt - \alpha_L)$ and induced forms
$\lambda_\pm' = \psi_*\lambda_\pm = (\psi^*)^{-1}\lambda_\pm$, 
let $F: W_{a_-}^{a_+} \to \R_+$
be a function defining a concordance structure between $\lambda_\pm$.
Then the function 
$G: W_{a_-}^{a_+} \to \R_+$, $(v, s) \mapsto F(\psi^{-1}(v), s)/h(\psi^{-1}(v))$
defines a concordance structure between $\lambda_\pm'$
where $h: V \to \R_+$ such that $\psi^*(dt - \alpha_L) = h(dt - \alpha_L)$
and $\omega_G$ is induced from $\omega_F$ by 
$\psi \times \Id: W_{a_-}^{a_+} \to W_{a_-}^{a_+}$.
This implies that $\textrm{mon}$ commutes with $\psi_\sharp$ on chain level.
The identical argument applies in the equivariant setting
assuming $\psi \in \GZKc$ and $\lambda_\pm$ is $\Z_k$-admissible.
Passing to homology,
 the morphism $\lambda_+ \preceq \lambda_-$ 
induces a morphism $\textrm{mon}: CH_*(\lambda_+) \to CH_*(\lambda_-)$
which
commutes with $\psi_\sharp$, and likewise in the equivariant case, 
for $\psi$ in the respective group.
In principle $\textrm{mon}$ still
depends on both the choice of 
concordance $F$ and of adjusted almost complex structure $J$, while
$CH_*(\lambda, J)$, $CH_*^{\Z_k}(\lambda, J)$ and $\textrm{win}$
depend on $J$. By a straightforward argument (see Proposition 4.30 of \cite{ekp}) 
which goes through verbatim
in our setting the vector spaces $CH_*(\lambda, J)$, $CH_*^{\Z_k}(\lambda, J)$
are independent of $J$.
Finally, by standard Floer theoretic arguments 
applied to 
a homotopy between concordances or between
almost complex structures adjusted to a concordance
 (see page 1692 of \cite{ekp}), it 
follows that monotonicity morphisms $\textrm{mon}$
are independent of the choices of $F$ and $J$.
We omit $J$ in the notation from now on.

\smallskip
{\bf Invariants of domains.}
Let $U \in \mathscr{U}_k$. 
The morphisms $\textrm{mon}$ 
make the family of vector spaces $\{CH^{\Z_k, (0, \epsilon)}(\lambda)\}_{\lambda \in \calF_{ad}(U, \epsilon)}$ into a directed
system over $\calF_{ad}(U, \epsilon), \preceq$ and after taking its direct limit,
the morphisms $\textrm{win}$ make 
resulting vector spaces $\{CH^{\Z_k, (0, \epsilon)}(U)\}_{\epsilon \in \R_+}$ into 
an inverse system over $\R_+, \geq$.
Following \cite{ekp} we define
$$CH^{\Z_k}(U) := \varprojlim\limits_{\epsilon \to 0} 
\varinjlim_{\lambda \in \calF_{ad}(U, \epsilon)} CH^{\Z_k, (0, \epsilon)}(\lambda)$$
and make the same definition without superscripts $\Z_k$ for 
$U \in \mathscr{U}$. Then given $U_1 \subset U_2$ there is an 
induced morphism $\iota_*: CH^{\Z_k}(U_1) \to CH^{\Z_k}(U_2)$
and the properties of \textrm{mon} and $\psi_\sharp$
pass to the double limit yielding the following statement in terms of 
$\calG$-functors\footnote{Repeating footnote 3 of \cite{ekp}:
Given a group $\calG$ acting on 
category $\mathscr{U}$, a $\calG$-functor on $\mathscr{U}$ 
is a functor $F$ and a family of 
natural transformations $g_\sharp:F \to F \circ g$, $g \in \calG$, 
such that $(gh)_\sharp = g_\sharp \circ h_\sharp$ for all $g, h \in \calG$; 
see also the reference to Jackowski and Slominska in \cite{ekp}},
stated for the 
non-equivariant setting on p. 1650 of \cite{ekp} (see also their Theorem 4.47):

\begin{thm}\label{thm:g-functor}
$CH^{\Z_k}(-)$ is a $\calG$-functor from $\mathscr{U}_k$ to $\Z$-graded vector spaces,
for $\calG = \GZKc$. 
\end{thm}

\smallskip
{\bf Technical arguments.} 
Whereas
Eliashberg-Kim-Polterovich \cite{ekp} restricted to $\Z_2$-coefficients 
in order to streamline their presentation  
avoiding the need for coherent orientations,
these can be produced
by the now-standard construction of Bourgeois-Mohnke \cite{fred-klaus} 
(see also Bourgeois-Oancea \cite{bourg-oanc07}) which
pulls back orientations from 
the orientations of determinant line bundles over certain spaces of Fredholm operators.
When all ingredients in the construction are $\Z_k$-invariant, 
the orientations on these bundles are as well, hence so too are
their pull-backs. This gives coherent orientations
in our equivariant setting and also means $CH_*(-)$ of \cite{ekp} is well-defined
with $\Z_k$-coefficients; we re-prove their results for $\preqB$
using $\Z_k$-coefficients, $k > 2$ in the next 
Section.

The compactification results used to define the differential in
\cite{ekp} are obtained as a consequence of the restriction on the action of
closed contractible Reeb orbits for admissible contact forms 
and the requirement that almost complex structures
be standard outside the symplectization of a compact subset $\Xi$ of
$V$. This second condition means $S\Xi$ is foliated
by weakly $J$-convex hypersurfaces so
all $J$-holomorphic cylinders project to $\Xi$. Given this and the 
absence of
contractible closed Reeb orbits in $\calP_\lambda$, the results from \cite{beh3} 
imply the needed compactification and glueing 
statements proving $d^2= 0$. The above goes through verbatim
in our equivariant setting as well.

Finally, we address transversality.
Because the $\Z_k$-action we consider is a covering space action,
only minor modifications of the standard arguments
for the non-equivariant setting of \cite{ekp} are required. 
By contrast, Milin's $\Z_k$-equivariant theory requires more intricate
arguments because her $\Z_k$-action, generated by
$(z, t) \mapsto (e^{2\pi i /k} z, t)$,
is not free (it fixes all points of $\{0\} \times S^1$).

The standard transversality arguments 
(see \cite{bourgeois} for a detailed
account) can be summarized
as follows. 
To show the moduli space $\hat \calM(\gamma_+, \gamma_-)
= \hat \calM_J(\gamma_+, \gamma_-)$ which
depends on a specific $J$ is, for generic $J$,
a Banach manifold of specified dimension,  
one considers, in the setting of \cite{ekp},
the universal moduli space $\hat \calM(\gamma_+, \gamma_-, \calI)$
where $\calI$ is the smooth Banach manifold of
almost complex structures adjusted to $\lambda$,
resp. adjusted to a concordance. 
This space, $\hat \calM(\gamma_+, \gamma_-, \calI)$, can be identified 
with the zero
set of the Cauchy-Riemann operator  $\overline \partial$
defined on the product $\calB \times \calI$, 
where $\calB$ is the smooth Banach manifold 
consisting of cylinder maps satisfying all conditions for elements
of $\hat \calM(\gamma_+, \gamma_-)$ except the Cauchy-Riemann
equation. By an infinite dimensional implicit function theorem,
it suffices to show that the linearized operator 
$L_{(G, J)} \overline \partial$ is surjective for all $(G, J)$ in
$\hat \calM(\gamma_+, \gamma_-, \calI)$, then it follows that 
$\hat \calM(\gamma_+, \gamma_-, \calI)$ 
 is a Banach manifold
and so $\hat \calM_J(\gamma_+, \gamma_-)$ 
is as well for all regular values $J$ of the
projection map $\pi:\hat \calM(\gamma_+, \gamma_-, \calI) \to \calI$,
while, by Sard-Smale, the regular $J$ are generic. In 
our $\Z_k$-equivariant setting we consider instead
analogous spaces $\calI^k$ consisting 
of $\Z_k$-adjusted almost 
complex structures. 
The only 
part of the above argument we must modify 
is the proof that $L_{(G, J)} \overline \partial$ is surjective
for all $(G, J)$ since in the $\Z_k$-equivariant setting 
this operator is now defined on the
smaller space $\calB \times \calI^k$ where
$\calI^k$ consists of $\Z_k$-adjusted almost complex structures.
In fact surjectivity here follows from that on $\calB \times \calI$
by a covering space argument for $\calI^k$ as a 
$k$-fold cover of $\calI$ but we instead describe
how the surjectivity argument on $\calB \times \calI$ can be directly modified.
 
Surjectivity in the specific case where
$\gamma_- = \gamma_+$ and $G$ is a vertical cylinder is immediate, in equivariant
as well as non-equivariant settings,
since $L_{(G, J)} \overline \partial$ decomposes as a direct sum of 
surjections (see \cite{bourgeois}).
For other $G$, in the non-equivariant setting 
surjectivity of $L_{(G, J)} \overline \partial$
is obtained by showing that for any $J \in \calI$,
every non-trivial, finite energy $J$-holomorphic cylinder $G = (f, a)$
in $W$ resp. $\overline W$ has
an {\it injective point} $p \in \R \times S^1$ \cite{dragnev}, \cite{bourgeois}. This is a point
$p$ such that $df_p \neq 0$ and 
$f^{-1}( f(p)) = \{p\}$. For genericity of almost complex
structures $J$ adjusted to concordances a similar
statement holds, and likewise for homotopies $J_r$, $r \in [0,1]$ 
of almost complex structures adjusted to
a single concordance (one can perturb to a homotopy of 
regular structures, fixing endpoints).
In all three cases, the existence of needed injective points 
is guaranteed by the fact that 
near the ends of each holomorphic cylinder 
$f$ restricts to an embedding \cite{hwz}.
One injective point implies a neighbourhood of such in every 
$J$-holomorphic cylinder $G$
and
this forces $L_{(G, J)} \overline \partial$ to be surjective. 
In all three cases, to pass from existence of an injective
point $p$ on every cylinder $G$
to surjectivity of $L_{(G, J)} \overline \partial$ for all 
$(G, J)$ the key ingredient is that one
can choose a tangent vector $X \in T_J\calI$
supported in a small ball around $F(p)$ with complete 
freedom due to injectivity of $f$ at $p$. This allows
to derive a contradiction should $L_{(G, J)}$ fail to
be surjective. 
In our $\Z_k$-equivariant framework where
we work with $\calI^k$, in order to 
carry this argument out it suffices to find a stronger kind
of injective point, one which also satisfies
$f^{-1}( \{f(p), \nu f(p), \ldots, \nu^{k-1} f(p)\}) = \{p\}$.
We call such a point {\it $\Z_k$-injective}\footnote{This definition 
 is motivated by
\cite{milin}, but she also requires that $f(p)$ avoid $\{0\} \times S^1$,
a consideration not needed in our case}.
Its existence 
means a tangent vector $X \in T_J\calI^k$
with support in $\cup_{j=0}^k \nu^j(B)$ can still be chosen with
complete freedom given a small ball $B$ at $p$ and so
the usual argument (see \cite{bourgeois}) 
goes through verbatim.

First we remark that when $G$ is $\Z_k$-equivariant for the
$\Z_k$-action on $\R \times S^1$ generated
by a $1/k$-shift in the
$S^1$ factor, then it is 
the lift to the $k$-fold cover of 
a $J_0$-holomorphic cylinder $G_0:\R \times S^1 \to V \times \R$
from the base of $\R \times S^1$
to the non-$\Z_k$-invariant base $V \times R$ 
of $\overline{W}$ resp. $W$
where $J_0$ is the non-$\Z_k$-invariant 
almost complex structure
in the base which lifts to $J$. We
know $L_{(G_0, J_0)} \overline \partial$
is surjective, so $L_{(G, J)} \overline \partial$ is as well.

We then consider $G$ which are not $\Z_k$-equivariant.
In this case, we claim the image of $G$ cannot be $\Z_k$-invariant.
Indeed if it were then
$G$ and $\nu G$ would be related by a holomorphic re-parametrization of
$(\R \times S^1, j)$, but this is necessarily a translation and amounts asymptotically
along $\gamma_\pm$ to a shift
by $1/k$, hence is exactly a shift by $1/k$, i.e. $G$ is
$\Z_k$-equivariant.
Now, because $G$ is not $\Z_k$-invariant we know
$\nu^j G(\R \times S^1) \neq G(\R \times S^1)$ for any $j \in U(k)$.
At the same time, intersection points of holomorphic curves can accumulate
only at critical values of both curves.
We take a point $p$ which is
injective for $G$ in the usual sense 
and consider a neighbourhood $U$ of $p$ consisting
also of injective points (by openness of this condition). Since
$G(U)$ can intersect the other holomorphic cylinders
$\nu^j G(\R \times S^1)$,  $j \in U(k)$ at only finitely many points,
there is necessarily a point $p'$ (and hence small neighbourhood
$U' \ni p'$) which is $\Z_k$-injective. 

\subsection{Results for $\preqB$: application to non-squeezing}\label{sec:results-preqB}
In this Section we state two results, Theorems~\ref{thm:comp-equiv} and~\ref{thm:psi-perturbed},
concerning $\Z_k$-equivariant contact homology of prequantized balls
and show that together they imply our main result, Theorem~\ref{thm:cover}. For
comparison, we also state Theorem~\ref{thm:comp-nonequiv}, a 
non-equivariant version of Theorem~\ref{thm:comp-equiv}
which is a direct analog with $\Z_k$-coefficients of Eliashberg-Kim-Polterovich's result 
for $CH_*(\preqB)$ with $\Z_2$-coef-ficients (Theorem 1.28 and page 1721 of \cite{ekp}, with different grading convention).
Proofs of these results are then given in Section~\ref{sec:groups-preqB}.
We remark that our
proof\footnote{Though we have formally restricted
to $k>2$ to simplify the presentation of $\Z_k$-equivariant proofs (see Remark~\ref{rem:Z2vsZk}) this non-equivariant argument
 goes through without
change for $\Z_2$-coefficients.}
of Theorem~\ref{thm:comp-nonequiv} 
uses only contact homology, and does not pass to generalized 
Floer homology as did the proof of Eliashberg-Kim-Polterovich; 
it thus provides an alternative,
more direct, proof of their result as well. We write $[s]$ for the integer part of $s \in (0, \infty)$.

\begin{thm}[c.f. Theorem 1.28 and page 1721 \cite{ekp}] \label{thm:comp-nonequiv}
When $1/R \not \in \N$
$$
CH_{m} (\preqB) =
\left \{ 
\begin{array}{ll}
\Z_k, &\text{ if }m = -n -2n[1/R] \\
0, &\text{ otherwise.}
\end{array}
\right .
$$
Moreover, if 
$[1/R_1] = [1/R_2]$  
then the morphism induced 
by inclusion $\preqB[2] \subset \preqB[1]$ is an isomorphism
in the grading $m = -n-2n[1/R_1]=-n-2n[1/R_2]$.
\end{thm}

By comparison, for $\Z_k$-equivariant contact homology we have:
\begin{thm}\label{thm:comp-equiv}
When $R > 1/k$ and $1/R \not \in \N$ 
$$CH_{m}^{\Z_k} (\preqB) = 
\left \{ 
\begin{array}{ll}
\Z_k, &\text{ if } m \geq -n -2n[1/R] \\
0, &\text{ if }m < -n -2n[1/R].
\end{array}
\right .
$$
Moreover, if 
$R_1 \geq R_2 > 1/k$ then the morphism induced 
by inclusion $\preqB[2] \subset \preqB[1]$ is an isomorphism in
all gradings $m \geq -n-2n[1/R_1]$.
\end{thm} 
Broadly speaking, the basic
computations for $CH_*(\preqB)$ and $CH_*^{\Z_k}(\preqB)$
are similar in that - in both cases - we find
cofinal sequences of contact forms $\lambda$
such that the chain complex $(C(\lambda)^{(0,\epsilon)}_*, d)$ 
is quasi-isomorphic to
$0 \to \Z_k \to 0$
with non-trivial chain module in degree $m_0 = -n -2n[1/R]$. 
A major difference in the two cases, however, 
is that non-equivariant homology 
$H_m(0 \to \Z_k \to 0)$ 
is non-trivial 
iff $m = m_0$, while $\Z_k$-equivariant homology 
$H_m^{\Z_k}(0 \to \Z_k \to 0)$
is non-trivial
iff $m \geq m_0$.

Not only can we conclude 
from Theorem~\ref{thm:comp-equiv} that 
the morphism $CH_{*}^{\Z_k} (\preqB) \to CH_{*}^{\Z_k} (\widehat B(R_\dagger))$
induced by an inclusion $\preqB \subset \widehat B(R_\dagger)$
is an {isomorphism in all degrees} for which both 
$CH_{*}^{\Z_k} (\preqB)$ and $CH_{*}^{\Z_k} (\widehat B(R_\dagger))$ are non-trivial, 
a stronger, $\psi$-perturbed, result also holds: 
\begin{thm}\label{thm:psi-perturbed}
Given $\psi \in \text{Cont}(\R^{2n} \times S^1)$ 
and $R_\dagger, R > 1$ such that
$\psi( \widehat{B} (R) ) \subset \widehat{B}(R_\dagger))$, the inclusion morphism
$CH_{*}^{\Z_k} (\psi( \widehat{B} (R) )) \to CH_{*}^{\Z_k} (\widehat {B}(R_\dagger))$
is an {isomorphism in all degrees} for which both 
$CH_{*}^{\Z_k} (\widehat{B} (R))$ and $CH_{*}^{\Z_k} (\widehat{B} (R_\dagger) )$ are non-trivial.
\end{thm}

Given the $\calG$-functoriality of
$CH_*^{\Z_k}(-)$ for $\calG = \GZKc$ (Theorem~\ref{thm:g-functor}),
Theorems~\ref{thm:comp-equiv} and~\ref{thm:psi-perturbed} together imply Theorem~\ref{thm:cover}:

\begin{proof}[{\it of Theorem~\ref{thm:cover}}.]
Suppose $\psi \in \GZK$ squeezes $\preqB[1]$ into $\preqB[2]$ and $R_2 < 1/\ell < R_1$.
By Remark~\ref{rem:into-itself} we may without loss of generality assume
$R_2$ to be as close to $1/\ell$ as desired, in particular, assume $1/k < R_2$.
By hypothesis $[\frac 1 {R_1}] <\ell < [\frac 1 {R_2}]$.

We have $\psi(\preqB[2]) \subset \psi(\preqB[1]) \subset \preqB[2]$. By Theorem~\ref{thm:psi-perturbed} the
inclusion morphism $CH_{p, \Z_k}(\psi(\preqB[2]) ) \to CH_{p, \Z_k}(\preqB[2])$
is an isomorphism for $p$ in the grading range $-n \geq p \geq -n-2n[\frac 1 {R_2}]$. 

Let $p = -n-2n\ell$. Then
$CH_{p, \Z_k}(\preqB[1]) = 0$ by Theorem~\ref{thm:comp-equiv};
hence, $CH_{p, \Z_k}(\psi(\preqB[1])) = 0$ by Theorem~\ref{thm:g-functor}.
This is a contradiction, as the isomorphism
$CH_{p, \Z_k}(\psi(\preqB[2]) ) \to CH_{p, \Z_k}(\preqB[2])$
must, by Theorem~\ref{thm:g-functor}, factor through $CH_{p, \Z_k}(\psi(\preqB[1]))$.
\end{proof}

\begin{rem}[Comparison with the argument of \cite{ekp}] \label{rem:compareNS}
The proof just given for Theorem~\ref{thm:cover} is in the same general spirit as 
Eliashberg-Kim-Polterovich's proof that $\preqB[1]$ cannot be 
squeezed into $\preqB[2]$ for $R_2 < 1 < R_1$. Their argument
(generalized
in Proposition 1.26 of \cite{ekp}) uses the non-vanishing of an
inclusion morphism $CH_{-n}(\psi(\preqB[1])) \to CH_{-n}(\preqBB[R_\dagger])$ 
for large\footnote{They need
$R_\dagger > 1$ such that $\preqBB[R_\dagger]$ contains the support of
$\psi$.} 
$R_\dagger$ and $CH_{-n}(\preqB[2]) = 0$ both of which
are given when $R_2 < 1 < R_1$. 
This strategy does not work verbatim in our setting
because 
$R_2 < R_1$ implies 
$CH_m^{\Z_k}(R_1) \neq 0 \Rightarrow CH _m^{\Z_k}(\preqB[2]) \neq 0$ in all gradings $m \in \Z$.
A squeezing $\psi$ of $\preqB[1]$ into $\preqB[2]$ implies, however, multiple interleavings,
$\psi(\preqB[2]) \subset \psi(\preqB[1]) \subset \preqB[2] \subset \preqB[1]$,
so one is not confined to have the ``smaller $R_j$" be the
``monkey-in-the-middle":  we
use the first pair of inclusions, while \cite{ekp} use essentially the final pair.
As we will see in Section~\ref{sec:groups-preqB}, 
Theorem~\ref{thm:psi-perturbed} which played a key role in our proof is
also very much
related to the isomorphism
$CH_{-n}^{\Z_k}(\psi(\preqB[2])) \to CH_{-n}^{\Z_k}(\preqBB[R_\dagger])$
which holds for large $R_\dagger$.
\end{rem}

\subsection{Computations for $\preqB$}\label{sec:groups-preqB}

\begin{proof}[{\it of Theorems~\ref{thm:comp-nonequiv} and~\ref{thm:comp-equiv}.}]
To prove all parts of these Theorems (and later 
Theorem~\ref{thm:psi-perturbed} as well) we use similar Hamiltonians. 
We first describe these in detail then give arguments for: (I) the first statements of both Theorems, 
(II) monotonicity morphisms when $[1/R_1]=[1/R_2]$, and (III) 
monotonicity morphisms when $[1/R_1] < [1/R_2]$.

\medskip
Let $R>0$, assume $\epsilon > 0$ and fix
$\delta \in (0, 1)$.
Let $m_0 := [1/R]$ (so $1/(m_0+1)R < 1 < 1/m_0R$)
and consider a 1-dimensional family\footnote{These 
functions are similar to ones defined by \cite{milin} but we fix $\delta$ 
and make other simplifications.}
 of piecewise-linear functions $F := F_{c}$, $c \in (1, \infty)$
as shown in Fig.~\ref{fig:onereeb} such that 
the value $b \in (0, 1)$ is determined by $c$ as specified below, and
\begin{figure}[h]
\begin{center}
\includegraphics[width=7cm]{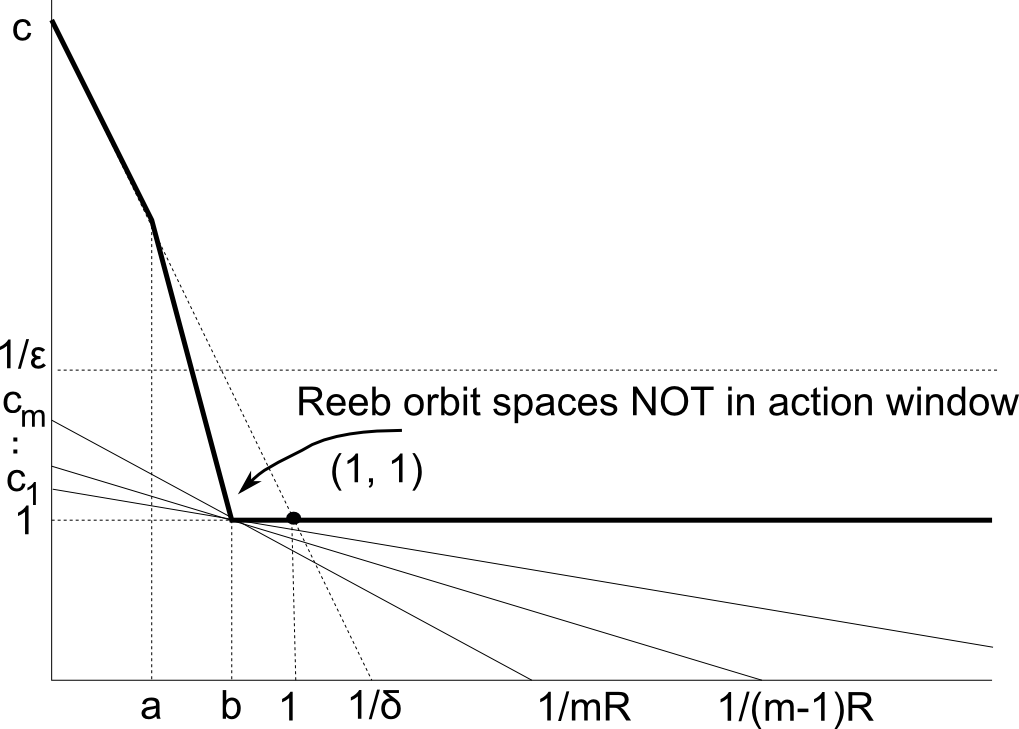} 
\end{center}
\caption{Graphs of Hamiltonians $F(u) := F_c(u)$ on $\preqB$ whose smoothings define 
contact forms $(dt - \alpha_L)/F$
with a single closed Reeb orbit at $\{0\} \times S^1$
of action $1/c$ and no other closed Reeb orbits with 
action in the window $(0, \epsilon)$. Here 
$m$ denotes $m_0 = [1/R]$ and
we assume $1< 1/\delta < 1/m_0R$ so there are no closed Reeb orbits
corresponding to tangencies at $(a, F(a))$. Orbits for $(b, F(b))$
on the other hand
do not contribute when $\epsilon$ is small since their action is
bounded below by $1/c_m$.} 
\label{fig:onereeb}
\end{figure}
$F$ consists of three linear pieces:
a left-most linear piece of slope $-c\delta$ which passes through $(0,c)$ and $(1/\delta, 0)$, 
a middle linear piece of twice that slope which passes through $(b, 1)$,
and a right-most piece, on the interval $(b, \infty)$ which is horizontal.
The
parameter $b$ is required to be an increasing function $b=f(c)$ of 
$c$ such that $\lim_{c \to \infty} f(c) = 1$,
and $f(c+1) < f(c) + 1/2(c+1)\delta$ for sufficiently large $c$.
For example $f(c) = 1 -1/c$ satisfies this inequality for $c > 2 \delta$.
A quick computation shows that the conditions on
the middle segment of $F_c$, namely its slope and choice of $b = f(c)$,
imply
$\frac 1 2 F_{c+1} \leq F_c$. Abusing notation, we define the Hamiltonian 
$F_c$ on $\R^{2n}$ to be the above function of the variable
$u = w/R$, for $w = \pi |z|^2$. This lifts
to an $S^1$-invariant Hamiltonian on $\R^{2n} \times S^1$
equal to $1$ outside
$\preqB$;
when clear from context
we will use the same name to refer to $F_c$ as a function of 
the real variable $u$ or as a Hamiltonian
on $\R^{2n}$ or $\R^{2n} \times S^1$. 
Put $\R_C := (C, \infty)$ 
for suitable $C > 2/\epsilon$ such that
$f(C) > 1/(m_0+1)R$. Then $f(c) > 1/(m_0+1)R$ for all
$c \in \R_C$ since the family $F_c$ is non-decreasing in $c$.
We consider 
the sequence $\{F_i\}_{i \in \N_{C}}$
where $\N_{C} := \N \cap \R_C$. 
Among Hamiltonians which are constantly equal to $1$ outside $\preqB$
this is a dominating (i.e., cofinal) set.

We will now
smooth each $F_i$ in small neighbourhoods
of its corners in such a way that the new tangent lines remain
``sufficiently close" to those of the original $F_i$ 
and the sequence of functions remains non-decreasing and dominating.
To make these conditions precise, 
assume the smoothing operator $A_\Delta$ has one parameter $\Delta> 0$ which
determines the fineness of the smoothing such that:
(1) by decreasing $\Delta$, $A_\Delta(F)$ becomes arbitrarily $C^0$-close to $F$,
(2) whenever $F$ is
linear on $(u - \Delta, u + \Delta)$ then
$A_\Delta(F)(u) = F(u)$, and 
(3) when 
$(u - \Delta, u + \Delta)$ contains a single corner of $F$,
$A_\Delta(F)(u) \geq F(u)$ with $A_\Delta(F)(u)$ convex in
the case $F$ is convex over the interval, and 
 $A_\Delta(F)(u) \leq F(u)$ with $A_\Delta(F)(u)$ concave 
 in the case of concativity. 
Consider $F :=F_i$ for some $i \in \N_{C}$.
For sufficiently small
$\Delta$, $F^A := A_\Delta(F)$ is a positive, non-increasing function
with $F^A(u)=c(1-\delta u)$ for small $u$ and $F^A(u) \equiv 1$
for $u \geq 1$. It is easily 
verified 
the contact form $\lambda_{F^A}  := (dt -\alpha_L)/F^A(u)$  
has one closed Reeb orbit
at $\{0\} \times S^1$ of action $1/F(0) = 1/c$ and 
spaces of closed Reeb orbits of action $1/c_m$ along 
spheres of constant $u$-value
such that $-(F^A)'(u)/(F^A(u) - u(F^A)'(u)) = mR$
for some $m \in \N \cup \{0\}$. We denote each such
$u$-value by $u_m$ and the sphere $\{u = u_m\}$ by $S_m$.
In this case, $c_m = -mR/(F^A)'(u_m)$. 
On the other hand, for small $\Delta$,
tangent lines to $F^A$ are well approximated by so-called
``generalized tangent lines" to the original piecewise linear
$F$: these are either true tangent lines at points $p$ of
linearity or else lines of slope $s \in [s_1, s_2]$ which pass
through a corner point $p$ where linear portions with 
slopes $s_1 < s_2$ meet.
As a result, one can read off the closed Reeb orbits of $\lambda_{F^A}$
and their approximate actions 
from the graph of the un-smoothed $F$ (c.f. Remark 5.2 \cite{milin}):
$F^A$ has one isolated orbit at $\{0\} \times S^1$
with action $1/c$ where $c$ is the vertical intercept of $F$,
and it has an $S^{2n-1}$-family of orbits 
with action arbitrarily close to $1/c_m$ corresponding to each generalized
tangent line to $F$ at $(b, F(b))$ that is horizontal ($m = 0$) or 
has horizontal intercept 
$1/mR > 0$ ($m = 1, 2, \ldots, m_0$) with $c_m$ being in either 
case the vertical intercept of the generalized tangent line, always less than 
$1/(1-m_0R)$.

If one imposes $1/\delta < 1/m_0R$ then 
(with $c > C$) this implies there are
no generalized tangent lines at $(a, F(a))$ with
horizontal intercept $1/mR$, $R \in \N$ and, thus,
no corresponding closed Reeb orbits (see Fig.~\ref{fig:onereeb}). 
Otherwise, in general,
there will be Reeb orbits corresponding
to $(a, F(a))$ and all have action greater than $c$ 
(see Fig.~\ref{fig:manyreeb}). Let $u = \overline u_m$ respectively
be the $u$-values where such tangencies to the smoothing occur, and
$S_m := \{u = \overline u_m\} \subset \R^{2n}$ the corresponding spheres.

We now 
require the $\Delta_i$ which defines a smoothing 
of $F_i$, $i \in \N_C$ 
to be sufficiently small that the true 
vertical intercepts of the tangent lines to the smoothing
have the same ordering and lie on the 
same side of $1/\epsilon$ as those of generalized tangent lines 
to $F_i$. Note we have fixed $\epsilon>0$ while
computing
$\varinjlim_{\lambda' \in \calF_{ad}(\preqB, \epsilon)} CH^{(0, \epsilon)}(\lambda')$. 
In addition $\Delta_i$ should
be sufficiently small that the $\Delta_i$-windows centred at corners of $F_i$ do not
overlap. 
Finally because $\{F_i\}_{i \in \N_{C}}$ is
non-decreasing we can take each $\Delta_i$ small enough 
that the smoothing $(F_i)^A := A_{\Delta_i}(F_i)$
will satisfy $(F_{i-1})^A \leq (F_i)^A \leq F_{i+1}$ 
and because $\frac 1 2 F_i \leq F_{i-1} \leq F_i$ we 
can further reduce $\Delta_i$ if necessary 
so that $\frac 1 2 (F_i)^A < (F_{i-1})^A \leq (F_i)^A$.
Inductively, the first condition implies
the sequence  $\{(F_i)^A\}_{i \in \N_{C}}$ is 
non-decreasing; the second condition will be used
to control monotonicity morphisms.

The above produces $SU(n)$-invariant Morse-Bott functions $(F_i)^A$
on $\R^{2n} = \C^n$ with
critical submanifolds $\{0\} \times S^1$ and the
spheres $S_m$ (if any).
We now perturb each $(F_i)^A$ following the Morse-Bott computational framework of Bourgeois 
\cite{bourgeois-toronto}
in order to obtain a Morse function $\widehat F_i$ on $\R^{2n}$ 
which is $\Z_k$-invariant near each $S_m$ (for a $\Z_k$-action to be specified) and
for which we know certain parts of the associated contact homology chain complex
$CH^{(0, \epsilon)}(\lambda_{\widehat F_i})$. 
The perturbing functions 
are constructed as follows.
Choose an arbitrarily $C^0$-small $\Z_k$-invariant Morse function
$g:S^1 \to \R$ with $k$ maxima $M_1, \ldots, M_{k-1}$
and $k$ minima  $m_1, \ldots, m_{k-1}$, $M_j = e^{2j\pi i/k}$
and $m_j = e^{(2j+1)\pi i/k}$, for the $\Z_k$-action generated
by rotation through $2\pi/k$. Now consider an arbitrarily $C^0$-small Morse function
$f: \CP^{n-1} \to \R$ with one critical point of each even index
$2j$, $j = 0, \ldots, (n-1)$ and its pullback $\pi^*f: S^{2n-1} \to \R$
via the Hopf bundle map $\pi:S^{2n-1} \to \CP^{n-1}$. This is 
a Morse-Bott function on $S^{2n-1}$ with critical submanifolds
which are isolated Hopf circles. Perturb $\pi^*f$ by adding 
to it a radially attenuated extension of $g$ supported
in a small 
tubular neighbourhood of each critical Hopf circle. This produces
a Morse function $h: S^{2n-1} \to \R$ with 
$k$ critical points each of index $2j$ and $2j+1$ for 
$j=0, \ldots, n-1$ and whose Morse complex
is 
\begin{diagram}
0&\rTo& \calR & \rTo^{p(T)} & \calR & \rTo^{\ (T-1)} & \calR 
&\rTo^{p(T)} & \cdots &\rTo^{\ (T-1)\ } &R & \rTo^{p(T)} 
& \ldots & \rTo^{(T-1)} & \calR &\rTo &0  
\end{diagram}
where each arrow is multiplication by the specified polynomial $(T-1)$ or
$p(T) = T^{k-1}+\ldots+T+1 \in \calR$. The function $h$ is $\Z_k$-invariant
for any\footnote{The reader comparing with 
Milin's argument should note that she uses only
$m = 1$; we need vary $m$ because our
$\Z_k$-action induces
on each $S_m$ a different $\Z_k$-action, namely that
generated by multiplication by $e^{2\pi i m/k}$.}
  $\Z_k$-action on $S^{2n-1} \subset \C^n$ generated by multiplication by
$e^{2\pi i m/k}$, $m \in U(k)$. Fix a particular choice of $m$ and
re-label the $2k$ critical points in every index if necessary so that 
permutation of
generators of the chain modules $\calR$ 
induced by
the chosen
$\Z_k$-action on $S^{2n-1}$
 is given by multiplication by $T$.
Then the above chain complex is
the $\Z_k$-equivariant Morse complex for $h$
for the chosen $\Z_k$-action on $S^{2n-1}$.
Let $\widehat F_i$ denote
the Morse function on $\R^{2n}$ which is
obtained by adding to $(F_i)^A$ a radially attenuated extension of $h$ supported
in a small tubular neighbourhood $\calN_m$ of each critical sphere $S_m$ of $(F_i)^A$.
On each $S_m \times S^1$ there are then
$2kn$ non-degenerate closed Reeb orbits for $\lambda_{\widehat F_i}$
which may be identified consistently with the $2kn$ critical points of $\widehat F_i$ on $S_m$.
The vertical shift contactomorphism $\nu$ on $\R^{2n} \times S^1$
moreover permutes these orbits in $S_m \times S^1$
 according to
multiplication by $e^{2\pi i m/k}$ of the corresponding 
critical points of $\widehat F_i$ in $S_m$. Since $[1/R] < k$,
the action is free, i.e., $m \in U(k)$ (this is important). As above, 
assume labeling of the $k$ critical points of each index on $S_m$
such that the associated permutation of these points as
generators $1, T, \ldots, T^{k-1}$ of the chain module $\calR$ 
 is given by multiplication by $T$.
On the small neighbourhood $\calN_m$ of each $S_m$, the function $\widehat F_i$
is a $\Z_k$-equivariant Morse function for the $\Z_k$-action 
given by multiplication by $e^{2\pi i m/k}$. By
 \cite{bourgeois-toronto} the Conley-Zehnder indices
of closed Reeb orbits for the contact form $\lambda_ {\widehat F_i}$ 
are determined by 
Morse indices of corresponding critical points of $\widehat F_i$ and the differentials in the
chain complex $(C^{(0, \epsilon)}, d)(\lambda_ {\widehat F_i})$ corresponding 
to concordances between critical points in a common sphere $S_m$ are
the same as the corresponding differentials
in the $\Z_k$-equivariant Morse complex of 
$\widehat F_i$. 
When $\epsilon>0$ is sufficiently small, so
that only closed Reeb orbits corresponding to $(0, F_i(0))$
and $(a, F_i(a))$ have action in the window $(0, \epsilon)$,
then letting $m_0 = [1/\delta]$
(which may be greater than $[1/R]$) and
$\ell \in \N \cup \{0\}$ be such that $m_0 + \ell = [1/R]$,
respective Conley-Zehnder indices of closed
Reeb orbits for $\widehat F_i$ are
$-n-2nm_0, 2n-2n(m_0+1), \ldots, 2n-2n(m_0+\ell)$ and
the chain complex $(C^{\Z_k, (0, \epsilon)}, d)(\lambda_ {\widehat F_i})$ is
therefore 
$0 \to \Z_k \to^{d_{m_0}} C[-n-2nm_0] \to^{d_{m_0+1}} C[-n-2n(m_0+1)] \to \cdots \to 
C[-n-2n(m_0+\ell)] \to 0$
where $C[j]$ denotes  a sub-complex 
which is the Morse complex for $h$
shifted in grading by $j$; only the $\Z_k$-equivariant 
differentials $d_{m+1}$ corresponding to concordances between 
Reeb orbits on adjacent spheres $S_m$ and $S_{m+1}$ are 
not immediately known. We assume once again
(as for the operator $A_\Delta$) that
the perturbations used to construct $\widehat F_i$
are suitably $C^0$-small and
carried out successively on $(F_i)^A$, $i \in \N_C$.

\medskip
(I) To simplify the notation put $\lambda_i := \lambda_{(F_i)^A}$.
These contact forms, which are $S^1$-invariant
hence $\Z_k$-invariant, constitute
a cofinal sequence not only for $\calF^{\Z_k}_{ad}(\preqB, \epsilon)$
but also for $\calF_{ad}(\preqB, \epsilon)$.
Moreover, each Morse-Bott function $(F_i)^A$ has only
one critical point, namely the origin, 
and so is already a $\Z_k$-invariant Morse function;
thus, $\widehat F_i = (F_i)^A$ and we retain the name $(F_i)^A$.
Since $\lambda_i$ has only the
closed Reeb orbit at $\{0\} \times S^1$ and this 
 has Conley-Zehnder index $-n$,
for all sufficiently small $\epsilon > 0$
the chain complex $(C^{(0, \epsilon)}, d)(\lambda_{i})$ 
for each $i \in \N_{C}$ is 
$0 \to \Z_k \to 0$ with non-trivial chain module in degree $-n$
 and hence
$CH_*^{(0, \epsilon)}(\lambda_{i})$ and
$CH_*^{\Z_k, (0, \epsilon)}(\lambda_{i})$ 
satisfy the formulae we want 
for $CH_*(\preqB)$, resp. $CH_*^{\Z_k}(\preqB)$.
It remains to check that monotonicity morphisms in both
equivariant and non-equivariant contact homology are isomorphisms
for chosen small $\epsilon > 0$; the above result for
specific $\lambda_{i}$ will then pass to the double-limit. We give the 
argument for the non-equivariant case to fix notation but the equivariant case is identical.
Fix $i \in \N_{C}$, and recall that
$\frac 1 2 (F_{i+1})^A \leq (F_{i})^A \leq (F_{i+1})^A$. 
To avoid excessive subscripts put $G = (F_{i+1})^A$. On the one hand,
there is a window enlargement isomorphism
$CH_*^{(0, \epsilon)}(\lambda_{\frac 1 2 G}) \xrightarrow{\cong} CH_*^{(0,  2\epsilon)}(\lambda_{\frac 1 2 G})$
because $\frac 1 2 G$ 
has no closed Reeb orbits with action in the window $[\epsilon, 2\epsilon]$ (recall $2/c < \epsilon$).
On the other hand, using scaling invariance, we also have an isomorphism
$CH_*^{(0, 2\epsilon)}(\lambda_{\frac 1 2 G}) = CH_*^{(0, 2\epsilon)}(2\lambda_{G}) \xrightarrow{\cong} CH_*^{(0, \epsilon)}(\lambda_{G})$. Composing these we obtain
the monotonicity morphism 
$\text{mon}: CH_*^{(0, \epsilon)}(\lambda_{\frac 1 2 G}) \to CH_*^{(0, \epsilon)}(\lambda_{G}) = 
CH_*^{(0, \epsilon)}(\lambda_{i+1})$
is an isomorphism. Since it factors through 
$\text{mon}: CH_*^{(0, \epsilon)}(\lambda_{i}) \to CH_*^{(0, \epsilon)}(\lambda_{i+1})$ and all
vector spaces $CH_m^{(0, \epsilon)}(\lambda_{i})$ are either trivial or $\Z_k$ this completes
the proof of the first statements in Theorems~\ref{thm:comp-nonequiv} and~\ref{thm:comp-equiv}.

\medskip
(II) To prove the second statement of Theorem~\ref{thm:comp-nonequiv} 
we take essentially the 
same sequence of Hamiltonians for both $\preqB[1]$ and $\preqB[2]$, however, as functions of their
respective canonical coordinates $u_1 := w/R_1$ and $u_2 := w/R_2$ where
$w := \pi |z|^2$.
Remark that
$u_1 = 1/mR_1$ $\Leftrightarrow$ $u_2 = 1/mR_2$ $\Leftrightarrow$
$w = 1/m$ and the condition $m_0 = [1/R_1] = [1/R_2]$ implies
$1/(m_0+1) < R_2 < R_1 < 1/m_0$ so 
$1/(m_0+1)R_j < 1 < 1/m_0R_j$ for both $j=1,2$.
Now let $\delta_0 > m_0$ but still less than $1/R_1 < 1/R_2$
and put $\delta_1 = \delta_0R_1$, $\delta_2 = \delta_0R_2$.
Then $1/\delta_j \in (1, 1/m_0R_j)$ for both $j=1, 2$. Define
Hamiltonians $F_c(u_1)$ for $\preqB[1]$ using $\delta = \delta_1$, 
and Hamiltonians $H_c(u_2)$ for $\preqB[2]$ using $\delta = \delta_2$. 
It follows that $H_c$ and $F_c$ both coincide
near $\{0\} \times S^1$ with the Hamiltonian $c(1-\delta_0 u)$.
On the other hand, assume the same function $f$ is chosen to define $b = f(c)$ 
in both families, $F_c$ and $H_c$. In the first case, 
$f$ gives values of $u_1$; in the second case of $u_2$. Since these are values
at which lower corners of the respective Hamiltonian occur
and $u = u_jR_j$, $j=1,2$ we see that $F_c$'s lower corner has
larger $u$-value than $H_c$'s and so $H_c \leq F_c$ for all $c$.
As in (I) we now consider only $c \in \R_C$ for suitable
$C > \epsilon/2$ such that $f(c) > 1/(m_0+1)R_2$. This 
implies $f(c) > 1/(m_0+1)R_1$ and if necessary we 
adjust\footnote{While $f$ takes values between $1/(m_0+1)R_2$ and $1$ 
we can re-parametrize it so that an integer value $N$ yields $f(N)$ arbitrarily
close to $1/(m_0+1)R_1 < R_2/R_1 <  1$ as desired.}
the definition of $f$ so that at least the minimal $N \in \N_C$ satisfies
$f(N) < R_2/R_1$. 
Now consider sequences $F_i, H_i$,
$i \in \N_C$ and assume smoothings are done
as above to yield $SU(n)$-invariant
Morse-Bott functions $(F_i)^A$ and $(H_i)^A$ respectively. 
These each have a single critical point and so are already
$\Z_k$-invariant Morse-functions.
Put $\lambda^F_i := \lambda_{(F_i)^A}$ and 
$\lambda^H_i := \lambda_{(H_i)^A}$.
By the same argument as
before, each $\lambda^F_i$ and each $\lambda^H_i$  
has the desired equivariant
and non-equivariant contact homology and these pass to the 
double limit. We claim that for fixed $\epsilon > 0$ sufficiently
small and $i \in \N_C$ sufficiently large the monotonicity morphism
induced by $\lambda^H_i \geq \lambda^F_i$ is an isomorphism, 
thus yielding the desired statement in the double limit. Indeed, 
$(F_N)^A \leq (H_i)^A$ for sufficiently large $i$ because 
$\{(H_i)^A\}_{i \in \N_C}$ is a dominating sequence for $\preqB[2]$
and $(F_N)^A(u_2)=1$
for all $u_2 \geq f(N)$ with $f(N) < R_2/R_1$
(recall $u_2 = R_2/R_1$ corresponds to $u = R_2$). 
At the same time
$\lambda^F_N \geq \lambda^F_i$ induces an isomorphism 
in contact homology for all
$i \in \N_C$ by the arguments above
so by functoriality and the fact that each 
$CH_m(\preqB)$ is either trivial or 1-dimensional, the
inclusion morphism $\iota_*: CH_m(\preqB[2]) \to CH_m(\preqB[1])$
is an isomorphism
in all gradings $m$, in particular, the grading $m = -n$ 
where both vector spaces are non-trivial. Note chain maps 
$C^{\Z_k, (0, \epsilon)}(\lambda^H_i) \otimes E \to C^{\Z_k, (0, \epsilon)}(\lambda^F_i) \otimes E$ 
induced by the $\Z_k$-equivariant chain maps 
$C^{\Z_k, (0, \epsilon)}(\lambda^H_i) \xrightarrow{\cong} C^{\Z_k, (0, \epsilon)}(\lambda^F_i)$ are then 
also isomorphisms in all gradings
so, taking the double limit, we obtain moreover the special case $[1/R_1] = [1/R_2]$
of the second statement of 
Theorem~\ref{thm:comp-equiv}.

\medskip
(III) Finally, to prove the second statement of Theorem~\ref{thm:comp-equiv} 
when $[1/R_1] < [1/R_2]$, consider
once again 
Hamiltonians as above for $\preqB[1]$ using $\delta = \delta_1$, 
and for $\preqB[2]$ using $\delta = \delta_2$, where
$\delta_1 /R_1 = \delta_0 = \delta_2/R_2$ for some choice of $\delta_0$.
Let $m_1= [1/R_1]$, $m_2 = [1/R_2]$.
\begin{figure}[h] 
\begin{center}
\includegraphics[height=5cm]{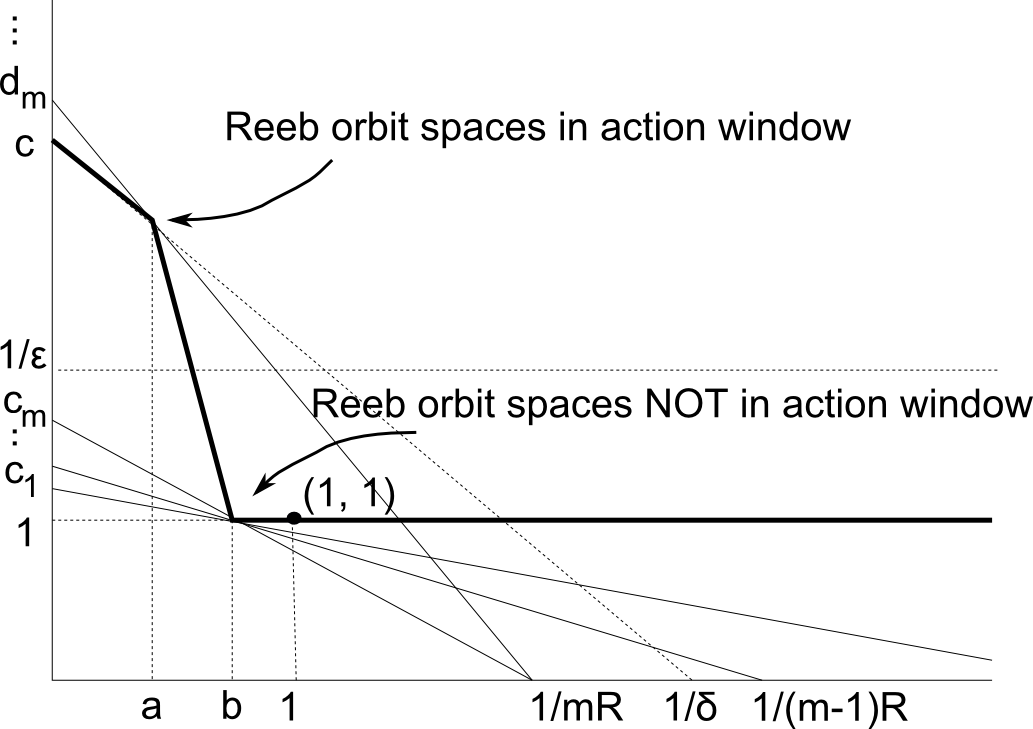} 
\end{center}
\caption{Graphs of Hamiltonians $H(u) := H_c(u)$ 
as in Fig.~\ref{fig:onereeb}
but where we allow $1/\delta > 1/m_0R$,
denoting $m_0 = [1/R]$ once again by $m$.
Smoothings of such $H$ define 
contact forms $\lambda = (dt - \alpha_L)/H$
which have not only the closed Reeb orbit $\{0\} \times S^1$
with action in the window $(0, \epsilon)$ as in Fig.~\ref{fig:onereeb},
but additionally
closed Reeb
orbits corresponding to generalized tangent lines at $(a, H(a))$;
unlike orbits for tangencies at $(b, H(b))$,
these orbits have action in the window $(0, \epsilon)$.
} \label{fig:manyreeb}
\end{figure} 
Now, unlike before, we can no longer simultaneously - for
$H_c$ and $F_c$ -
prevent generalized tangent lines at 
the upper corners which have horizontal intercept of the form $1/mR$. Indeed,
taking $\delta_0 > m_1$ (but less than $1/R_1$) and $C$ sufficiently large 
we prevent such lines for $F_c$ as in Fig.~\ref{fig:onereeb}, but then
we necessarily have 
$\delta_0 \not > m_2$ 
(because $1/R_1 < 1/R_2 < m_2$)
so such tangent lines will occur for $H_c$ as in Fig.~\ref{fig:manyreeb}.
We can guarantee, however, that their horizontal intercepts $1/mR_2$ occur only for $m$
such that
$m_1 < m < m_2$ by taking 
$\delta_0 > m_1$ and restricting
 to $c \in \R_C$ for some $C > \epsilon/2$ such that 
$f(c) > 1/(m_2+1)R_2$. Note, for $m_1 = m_2$ these conditions reduce to those
imposed in paragraph (II). 
As before, take sufficiently large $C$ and now
consider perturbed smoothings 
$\{\widehat F_i\}_{i \in \N_C}$ and $\{\widehat H_i\}_{i \in \N_C}$
which are $\Z_k$-invariant Morse functions. Note: unlike in (II),
the functions
$(H_i)^A$ were Morse-Bott with critical submanifolds $S_m$, $m_1 < m < m_2$. Note $m_2 < k$.
Now, fix $\epsilon > 0$ and take $i \in \N_C$ sufficiently large that 
$F_i(0) = H_i(0) > 1/\epsilon$.  
Let $\eta > 0$ be a scaling factor such that 
the line $y = \eta(1-\delta u)$ is tangent to and lies below $\widehat H_i$.
Then $\eta \widehat F_i \leq \widehat H_i$ and since $1/\delta_1 < 1/m_1R_1$,
the action of all Reeb orbits corresponding to
tangent lines to $\eta \widehat F_i$ at points near the
former ``lower corner" of $\widehat H_i$ will be strictly lower that
$\eta \widehat F_i(0)$. Let $\epsilon_0 < \eta \widehat F_i(0)$ be greater than
these actions. Note that 
$\textrm{mon}: C_*^{\Z_k, (0, \epsilon_0)}(\eta \widehat F_i) \to C_*^{\Z_k, (0, \epsilon_0)}(\widehat F_i)$
is then an isomorphism (by composing window enlargement 
and scaling invariance isomorphisms) and must factor through
$C_*^{\Z_k, (0, \epsilon_0)}(\widehat  H_i)$ by functoriality. This implies 
the monotonicity morphism
$\textrm{mon}: C_j^{\Z_k, (0, \epsilon_0)}(\widehat  H_i) \to C_j^{\Z_k, (0, \epsilon_0)}(\widehat  F_i)$
is an isomorphism in grading $j = -n - 2n m_1$ since all
three chain modules are $\Z_k$ in that grading. Note that
shrinking $\epsilon$ does not affect monotonicity morphisms
except as generators of the respective chain modules appear
or disappear. Since when we shrink $\epsilon \to 0$ no 
closed Reeb orbits of Conley-Zehnder index $-n-2nm_1$
appear or disappear in chain modules $C_j^{\Z_k, (0, \epsilon_0)}(\widehat  H_i)$
and 
$C_j^{\Z_k, (0, \epsilon_0)}(\widehat  F_i)$, $j = -n - 2n m_1$
we conclude that 
$\textrm{mon}: C_j^{\Z_k, (0, \epsilon)}(\widehat  H_i) \to C_j^{\Z_k, (0, \epsilon)}(\widehat  F_i)$
is an isomorphism 
too in this grading. 
Because $C_m^{\Z_k, (0, \epsilon)}(\widehat  F_i) = 0$ in all other gradings $m$,
$j$ is the only grading in which the chain map $\textrm{mon}$
is non-trivial and by reasoning as in (II), one checks that 
the induced chain map $C_*^{\Z_k, (0, \epsilon)}(\lambda^H_i) \otimes E 
\to C_*^{\Z_k, (0, \epsilon)}(\lambda^F_i) \otimes E$
consists of isomorphisms in row $j$ of the double complex 
$C_*^{\Z_k, (0, \epsilon)}(\lambda^H_i) \otimes E$ and vanishes in all other rows.
 This implies 
$\textrm{mon}: CH^{\Z_k, (0, \epsilon)}_m(\widehat  H_i) \to CH^{\Z_k, (0, \epsilon)}_m(\widehat  F_i)$
is an isomorphism in all gradings $m \geq -n-2nm_1$. 
Since this holds for all
$\epsilon > 0$ and $i$ sufficiently large it passes to 
the double limit, implying 
the inclusion morphism $\iota_*: CH^{\Z_k}_m(\preqB[2]) \to CH^{\Z_k}_m(\preqB[1])$
is an isomorphism for $m \geq -n-2nm_1$.
\end{proof}

To prove Theorem~\ref{thm:psi-perturbed} we use the following:
\begin{lem} \label{lem:zk-induction}
Assume a $\Z_k$-chain map
\begin{diagram}
0&\rTo& \Z_k & \rTo^{c_0^1p} & \calR & \rTo^{\ (T-1)} & \calR 
&\rTo^{c_1^1p} & \cdots &\rTo^{\ (T-1)\ } &R & \rTo^{c_{N}^1p} 
& \ldots & \rTo^{(T-1)} & \calR &\rTo &0 \\
&&\dTo{a_0}&&\dTo{a_1}&&\dTo{a_2}&&&&\dTo{a_{2N}}\\
0&\rTo& \Z_k & \rTo^{c_0^2p} & \calR & \rTo^{\ (T-1)} & \calR 
&\rTo^{c_1^2p} & \cdots &\rTo^{\ (T-1)} &\calR & \rTo & 0 
\end{diagram}
where each arrow is multiplication by the specified field element $a_0 \in \Z_k$, 
or polynomial $(T-1) \in \calR$, 
$a_j = a_j(T) \in \calR$ or
$c_i^j p \in \calR$ for $p = p(T) = T^{k-1}+\ldots+T+1 \in \calR$ and $c_i^j \in U(k)$, $N=nm_1$. 
{Then, if $a_0$ is a unit of $\Z_k$, all $a_j$, $j \in\{1,\ldots, 2N\}$ are units of $\calR$.}
Moreover, the conclusion also holds if the direction of the vertical arrows is reversed
but all other hypotheses are the same.
\end{lem}

This follows by 
induction\footnote{The author first encountered this observation
in Milin \cite{milin}.}
 on $j \in \N \cup \{0\}$ using the commutativity of 
$\calR$ and the fact that $(T-1)$, the maximal ideal of the local ring $\calR$, contains
all non-units and annihilates $p(T)$.

\begin{proof}[{\it of Theorem~\ref{thm:psi-perturbed}}.]
Fix $\epsilon > 0$.
Let $R_0 >  1$ be greater than 
both $R$ and $R_\dagger$ and
also sufficiently large that $\preqB[0]$
contains
the support of $\psi$. By Theorem~\ref{thm:comp-equiv}
the monotonicity morphism 
${\rm mon}: CH_{-n}(\preqB) \to CH_{-n}(\preqB[0])$
is an isomorphism. Therefore by functoriality (Theorem~\ref{thm:g-functor})
${\rm mon}: CH_{-n}(\psi(\preqB)) \to CH_{-n}(\preqB[0])$ is an isomorphism
and hence, using the fact these vector spaces 
are all equal to $\Z_k$, 
${\rm mon}: CH_{-n}(\psi(\preqB)) \to CH_{-n}(\widehat{B}(R_\dagger))$ is an
isomorphism too. 

Assume dominating sequences of smoothed,
perturbed Hamiltonians 
$\{\widehat{H}_i\}_{i \in \N}$, $\{\widehat{H'}_i\}_{i \in \N}$ and $\{\widehat{F}_i\}_{i \in \N}$
for $\preqB$, $\widehat{B}(R_\dagger)$ and $\preqB[0]$ respectively
 as in part (III) of the previous proof where now the role of $R_1$ is 
 played by $R_0$ for which $0 = [1/R_0]$.
 All contact forms 
 $\lambda^H_i$, $\lambda^{H'}_i$ and $\lambda^F_i$ therefore have
 a closed Reeb orbit at $\{0\} \times S^1$ with Conley-Zehnder index $-n$;
the forms $\lambda^F_i$ have only this orbit, while the forms 
 $\lambda^H_i$, $\lambda^{H'}_i$ have others as well. 
 More precisely, from the discussion
 of Morse-Bott computations in the previous proof (c.f. Bourgeois \cite{bourgeois-toronto})
 the chain complexes 
 $C_m^{\Z_k, (0, \epsilon)}(\lambda^ H_i)$ and $C_m^{\Z_k (0, \epsilon)}(\lambda^{H'}_i)$
 are of the form $0 \to \Z_k \to^{d_{0}} C[-n] \to^{d_{1}} C[-n-2n] \to \cdots \to^{d_{\ell}}
C[-n-2n\ell] \to 0$ for $\ell$ respectively given by $[1/R]$ or $[1/R_\dagger]$.
Moreover, by $\calG$-invariance, $\calG = \GZKc$, 
$C_m^{\Z_k, (0, \epsilon)}(\psi_*\lambda^H_i)$ is of the same form.
By comparison with the known non-equivariant contact homology of the
respective domains one deduces that for sufficiently large $i$, all maps $d_j$ are multiplication
by $c^j p(T) \in \calR$ for $p(T) = T^{k-1}+\ldots+T+1 \in \calR$ and $c^j \in U(k)$.

By shifting the indexing of the sequence $\{H_i'\}_{i \in \N_C}$ if necessary
we may assume that ${H'}_i \geq H_i \circ \psi$, i.e., $\lambda_i^{H'} \leq \psi_*\lambda_i^{H}$.
There is thus a corresponding induced $\Z_k$-equivariant chain map 
$\textrm{mon}: C_*^{\Z_k, (0, \epsilon)}(\psi_*\lambda_i^{H}) \to C_*^{\Z_k, (0, \epsilon)}(\lambda_i^{H'})$
which will be as shown in Lemma~\ref{lem:zk-induction}, but with all vertical arrows
possibly directed upwards.
Since $\textrm{mon}: CH_{-n}(\psi(\preqB)) \xrightarrow{\cong} CH_{-n}(\widehat{B}(R_\dagger))$
is an isomorphism and the vector spaces 
$C_m^{\Z_k, (0, \epsilon)}(\lambda_i^{H'})$ and $C_m^{\Z_k, (0, \epsilon)}(\psi_*\lambda_i^{H})$ 
are $\Z_k$ for all sufficiently large $i$
the chain map $\textrm{mon}$ in degree $-n$ must be an isomorphism for 
all sufficiently large $i$ and the result follows 
by applying Lemma~\ref{lem:zk-induction} and passing to the double limit.
\end{proof}

\section{Squeezing room and non-squeezing}\label{sec:room}

By using/extending two of the constructions in 
Eliashberg-Kim-Polterovich \cite{ekp} - namely, the map $F_N: \R^{2n} \times S^1 \to \R^{2n} \times S^1$ defined for $N \in \Z$ (see page 1649 of \cite{ekp}),
$$
(z, t) \mapsto (v(z)e^{2\pi N it}z, t), \hspace{.5cm}
v(z) = \frac{1}{\sqrt{1 + N\pi |z|^2}}
$$
and the squeezing given in Theorem 1.19 of \cite{ekp}
by a positive contractible loop of contactomorphisms
of the ideal contact boundary $P = S^{2n-1} \subset \R^{2n}$ - two observations follow.

First, that non-squeezing past $R = m/\ell > 1$
is equivalent to a ``squeezing room" requirement for
squeezing past $R = m/\kappa < 1$ which is in some cases stronger 
than what is proved in \cite{ekp}\footnote{Statements in this Section
for $R = m/\kappa$ are to be compared with statements in
\cite{ekp} for $R = m/k$; we have used the letter $\kappa$ instead of $k$ to
avoid confusion with notation of preceding 
Sections. Considering the effect of the map $F_b$, $\kappa$ should be thought of 
as $\ell + bm$ in this sentence.}.
Second, even this stronger requirement may not be tight.

To see the first observation, 
note that $F_b$ maps $\widehat B(R)$ into $\widehat B(R /(1 + bR))$
for all $R > 0$ and so in particular maps all of $\R^{2n} \times S^1$
into $\widehat B(1/b)$, taking $\widehat B(R)$ with $R \in [1, \infty)$
to $\widehat B(R')$ with 
$\R'\in [\frac 1 {b+1}, \frac 1 {b})$. Thus, 
Theorem~\ref{thm:main}, resp. \ref{thm:main-strong} is equivalent to
\begin{thm}\label{thm:main2}
When
\begin{align}
\frac 1 {b+1} < &\frac m \kappa < \frac 1 {b} \label{eq:sandwich}
\end{align}
there is no squeezing (resp. coarse squeezing)
of $\widehat B(m/\kappa)$ into itself within $\widehat B(1/b)$.
\end{thm}
In this sense, rigidity at large scale 
which completely precludes squeezing can be thought of as an infinite squeezing room 
requirement, and is equivalent to a form of rigidity at small-scale which requires
squeezing room determined
by the reciprocal integers. With this viewpoint, instead of a
single cut-off between flexibility and rigidity of $\preqB$ as $R$ grows, one sees rather a 
squeezing room requirement which jumps at each reciprocal integer, culminating in 
an infinite requirement when $R=1$ is passed.

Theorem 1.5 of \cite{ekp} established that there 
is no squeezing of $\widehat B(m/\kappa)$ into itself within 
$\widehat B(m/(\kappa-1))$ but this bound depends on the particular $m/\kappa$.
If equation \eqref{eq:sandwich}
holds and $m/\kappa$ is not too close to $1/b$ - more precisely, if
\begin{align}
\frac 1 {b+1} < &\frac m \kappa < \frac 1 {b + 1/m} \label{eq:strongsand}
\end{align}
then we will have $mb + 1 < \kappa$, i.e.,
\begin{align*}
&\frac m {\kappa-1} < \frac 1 b
\end{align*}
so the squeezing room requirement of Theorem~\ref{thm:main2}
will (for such $m/\kappa$) be strictly stronger than that of Theorem 1.5 of \cite{ekp}.

We now remark that even
this stronger squeezing room requirement is not necessarily tight.
If one applies Theorem 1.19 of \cite{ekp} 
to $\widehat B(m/\kappa)$,  $m < \kappa$ as done in Remark 1.23 of \cite{ekp} 
to $\widehat B(1/\kappa)$, one obtains the more general:
\begin{prop}
There is a squeezing
of $\widehat B(m/\kappa)$ into itself within an arbitrarily small neighborhood
of $\widehat B(m/(\kappa-m))$. 
\end{prop}
However,
\eqref{eq:sandwich} implies $m/(\kappa-m) > 1/b$ so there 
remains a gap between the required squeezing room
and that of known squeezings. To find squeezings with
smaller support it seems 
methods
beyond the construction of \cite{ekp} would be needed,
to produce either ``wilder" compactly supported 
contact isotopies which do not deform the original fiberwise
star-shaped domain through fiberwise star-shaped domains, or 
compactly supported contactomorphisms not isotopic to 
the identity.

\paragraph{Future work}
Besides closing the gap mentioned above 
another direction for future work is application of
the framework of the present paper to prequantizations $M \times S^1$ 
of other Liouville manifolds $M$. 
When $M$ is a sufficiently stabilized Liouville manifold, squeezing at 
small scale has already been established by Eliashberg-Kim-Polterovich
(c.f. Theorems 1.16, 1.19 in \cite{ekp} of which Theorem 1.1 
in our paper re-states the 
special case $M = \R^{2n}$, $n \geq 2$).
We conjecture that coarse non-squeezing
at large scale holds as well in these settings.

\paragraph{Acknowledgements}
I would like to express my deep gratitude to
Leonid Polterovich for his mathematical guidance during 
several semesters of my studies (in Computer Science)
at the University of Chicago, and for introducing me to 
the contact non-squeezing problem.
None of the work in this paper would have started without the 
many discussions with him and this work became
something of a second PhD for me, bringing me back into Mathematics.
In particular the idea to tackle this problem 
by developing a $\Z_k$-equivariant form of contact homology
is due to him. I owe 
a second
debt of gratitude to Isidora Milin whose PhD thesis on
orderability of lens spaces was an invaluable 
resource to me in developing the $\Z_k$-equivariant theory 
in this paper and 
also revealed enlightening similarities between 
the seemingly unrelated 
non-squeezing phenomena considered in her setting and ours. I also thank
Fr\'ed\'eric Bourgeois for very helpful technical discusssions.
Finally I thank Yasha Eliashberg.
His constant support and encouragement
kept this project going despite countless delays.
A portion of the writing of
this work was completed while visiting Tel Aviv University in May 2014 and the Universit\'e de
Lyon in June-July 2014. I
thank both Leonid Polterovich and Jean-Yves Welschinger for their hospitality
and for stimulating discussions. 
Those visits were partially supported by TAU and the CNRS respectively.

\end{document}